\documentclass[10pt,a4paper]{article}
\usepackage[a4paper]{geometry}
\usepackage{amssymb,latexsym,amsmath,amsfonts,amsthm,scalerel,mathtools}
\usepackage{graphicx,comment}
\usepackage{epsfig}
\usepackage{color}
\usepackage[hyperindex,breaklinks]{hyperref}
\usepackage{tikz}
\usepackage{etoolbox}
\usepackage[T1]{fontenc}

\newcommand{\Tr}{\mathrm{Tr}\,}

\newtheorem{theorem}{Theorem}[section]
\newtheorem{lemma}[theorem]{Lemma}
\newtheorem{proposition}[theorem]{Proposition}

\theoremstyle{definition}

\theoremstyle{remark}

\newtheorem{remark}[theorem]{Remark}

\numberwithin{equation}{section}

\title{Matrix continued fractions associated with lattice paths, resolvents of difference operators, and random polynomials}

\author{J. Kim  \qquad A. L\'{o}pez-Garc\'{i}a \qquad V.A. Prokhorov}

\date{\today}

\begin{document}

\maketitle

\begin{abstract}
We begin our analysis with the study of two collections of lattice paths in the plane, denoted $\mathcal{D}_{[n,i,j]}$ and $\mathcal{P}_{[n,i,j]}$. These paths consist of sequences of $n$ steps, where each step allows movement in three directions: upward (with a maximum displacement of $q$ units), rightward (exactly one unit), or downward (with a maximum displacement of $p$ units). The paths start from the point $(0,i)$ and end at the point $(n,j)$. In the collection $\mathcal{D}_{[n,i,j]}$, it is a crucial constraint that paths never go below the $x$-axis, while in the collection $\mathcal{P}_{[n,i,j]}$, paths have no such restriction. We assign weights to each path in both collections and introduce weight polynomials and generating series for them. Our main results demonstrate that certain matrices of size $q\times p$ associated with these generating series can be expressed as matrix continued fractions. These results extend the notable contributions previously made by P. Flajolet \cite{Fla} and G. Viennot \cite{Viennot} in the scalar case $p=q=1$. The generating series can also be interpreted as resolvents of one-sided or two-sided difference operators of finite order. Additionally, we analyze a class of random banded matrices $H$, which have $p+q+1$ diagonals with entries that are independent and bounded random variables. These random variables have identical distributions along diagonals. We investigate the asymptotic behavior of the expected values of eigenvalue moments for the principal $n\times n$ truncation of $H$ as $n$ tends to infinity.

\smallskip

\textbf{Keywords:} Matrix continued fraction, lattice paths, difference operator, banded matrix, random polynomials.

\textbf{MSC 2020:} Primary 05A15, 30B70, 47A10, 60G99; Secondary 41A20.
\end{abstract}

\tableofcontents

\section{Introduction}

As the title indicates, in this paper we explore different topics in several areas of research. We study lattice paths, generating series, and matrix continued fractions. We also investigate questions on difference operators, random matrices and their characteristic polynomials, and rational approximation. The common thread that connects these various subjects is the use of banded matrices and difference operators.

In this introduction we give an overview of the problems investigated and the main results obtained. Key ideas and the central objects of our study are underlined. For the sake of brevity, the more technical details, and occasionally the precise definitions, are left to the discussion in the subsequent sections.  

\subsection{Lattice paths, resolvent functions of difference operators, and matrix continued fractions}

First, we define two infinite banded matrices $H$ and $W$, which will play a central role in our investigation. The entries of these matrices will be used to introduce weights (or labels) in certain collections of lattice paths. 
\par
Fix two arbitrary integers $p, q\geq 1$, which throughout the paper will remain fixed. For each $-p\leq k\leq q$, let $(a^{(k)}_{n})_{n \in \mathbb{Z}}$ be a bi-infinite sequence of complex numbers. Let $H=(h_{i,j})_{i,j=0}^{\infty}$ denote the infinite matrix with entries
\begin{equation}\label{def:entriesH}
\begin{cases}
h_{j,j+k}=a_{j}^{(k)}, & 0\leq k\leq q, \quad j\geq 0,\\[0.3em]
h_{j+k,j}=a_{j}^{(-k)}, & 0\leq k\leq p, \quad j\geq 0,\\[0.3em]
h_{i,j}=0, & \mbox{otherwise}.
\end{cases}
\end{equation}
Note that this is a banded matrix with  $q$ superdiagonals and $p$ subdiagonals. The matrix $H$ has zero entries below the $p$-th subdiagonal, and zero entries above the $q$-th superdiagonal:
\begin{equation}\label{matrixH}
H=\begin{pmatrix}
a_{0}^{(0)} & \cdots & a_{0}^{(q)} & & & \\
\vdots & a_{1}^{(0)} & \cdots & a_{1}^{(q)} & &\\
a_{0}^{(-p)} & \vdots & a_{2}^{(0)} & \cdots & a_{2}^{(q)} &\\
 & a_{1}^{(-p)} & \vdots & \ddots & & \ddots \\
 &  & a_{2}^{(-p)} & & \ddots \\
 & & & \ddots &  
\end{pmatrix}.
\end{equation}
Let $W=(w_{i,j})_{i,j \in \mathbb{Z}}$ denote the two-sided infinite matrix with entries
\begin{equation}\label{def:entriesW}
\begin{cases}
w_{j,j+k}=a_{j}^{(k)}, & 0\leq k\leq q, \quad j \in \mathbb{Z},\\[0.3em]
w_{j+k,j}=a_{j}^{(-k)}, & 0\leq k\leq p, \quad j\in \mathbb{Z},\\[0.3em]
w_{i,j}=0, & \mbox{otherwise}.
\end{cases}
\end{equation}
The matrix $W$ is also a banded matrix with $q$ superdiagonals and $p$ subdiagonals, and observe that $H$ is a submatrix of $W$:
\begin{equation}\label{matrixW}
W=
\left(\begin{array}{ccccccccc}
  & & \ddots & & & & & & \\ 
& \ddots & & a_{-q-1}^{(q)} & & & & & \\
\ddots & & \ddots & \vdots & a_{-q}^{(q)} & & & & \\
& a_{-p-1}^{(-p)}& \cdots & a_{-1}^{(0)} & \vdots & \ddots & & & \\[0.5em]
&  & a_{-p}^{(-p)} & \cdots & a_{0}^{(0)} & \cdots  & a_{0}^{(q)} &  & \\
& & & \ddots &\vdots & a_{1}^{(0)} & \cdots & a_{1}^{(q)} & \\
& &  & & a_{0}^{(-p)} & \vdots & \ddots & & \ddots \\
& &  &  & & a_{1}^{(-p)} & & \ddots & \\
&  & &   &  & & \ddots & &     
\end{array}\right).
\end{equation}
In this paper we define two main collections of lattice paths, which are denoted $\mathcal{D}_{[n,i,j]}$ and $\mathcal{P}_{[n,i,j]}$. In this notation, $n$ indicates the length of the paths in the collection, and $i,j$ indicate the $y$-coordinates of the initial and terminal points of the paths, respectively. We then introduce the corresponding generating series $A_{i,j}(z)$ and $W_{i,j}(z)$, associated with the collections $\{\mathcal{D}_{[n,i,j]}\}_{n\geq 0}$ and $\{\mathcal{P}_{[n,i,j]}\}_{n\geq 0}$, respectively. The primary goal of our study is to establish two main results, namely that certain matrices of size $q\times p$ with entries given by these generating series can be expressed as matrix continued fractions. 

Each path in the collections $\mathcal{D}_{[n,i,j]}$ and $\mathcal{P}_{[n,i,j]}$ consists of a sequence of $n$ steps with vertices in the lattice $\mathbb{Z}_{\geq 0}\times\mathbb{Z}$, with initial point $(0,i)$ and terminal point $(n,j)$. Each step can be either upward (with a maximum displacement of $q$ units), horizontal (exactly one unit to the right), or downward (with a maximum displacement of $p$ units). In the collection $\mathcal{D}_{[n,i,j]}$ we also require that paths cannot go below the $x$-axis, with no such restriction for the paths in the collection $\mathcal{P}_{[n,i,j]}$.

We assign a weight $w(e) = a_m^{(k)}$ to each step $e$ in a lattice path, where $m$ is the $y$-coordinate of the lowest point in the step, and $-p\leq k\leq q$ indicates the difference between the $y$-coordinates of the terminal and initial points of $e$. The weight of a lattice path $\gamma$ is then defined as the product of the weights of all the steps in the path, i.e., we define
\[
w(\gamma)=\prod_{e\subset\gamma}w(e),
\]
where the product runs over the different steps of $\gamma$. In Figs.~\ref{7-1-22-1}--\ref{7-1-22-2} we illustrate some examples of paths and their weights.

\begin{figure}
\begin{center}
\begin{tikzpicture}[scale=0.6]
\draw[line width=1.5pt]  (-3.03,0) -- (15.5,0);
\draw[line width=1.5pt]  (-3,0) -- (-3,4.5);
\draw[line width=0.7pt] (-3,0) -- (-2,1) -- (-1,3) -- (0,1) -- (1,4) -- (2,1) -- (3,0) -- (4,2) -- (5,3) -- (6,3) -- (7,4) -- (8,0) -- (9,2) -- (10,1) -- (11,0) -- (12,3) -- (13,1) -- (14,2) -- (15,3);
\draw [line width=0.5] (-2,0) -- (-2,-0.12);
\draw [line width=0.5] (-1,0) -- (-1,-0.12);
\draw [line width=0.5] (0,0) -- (0,-0.12);
\draw [line width=0.5] (1,0) -- (1,-0.12);
\draw [line width=0.5] (2,0) -- (2,-0.12);
\draw [line width=0.5] (3,0) -- (3,-0.12);
\draw [line width=0.5] (4,0) -- (4,-0.12);
\draw [line width=0.5] (5,0) -- (5,-0.12);
\draw [line width=0.5] (6,0) -- (6,-0.12);
\draw [line width=0.5] (7,0) -- (7,-0.12);
\draw [line width=0.5] (8,0) -- (8,-0.12);
\draw [line width=0.5] (9,0) -- (9,-0.12);
\draw [line width=0.5] (10,0) -- (10,-0.12);
\draw [line width=0.5] (11,0) -- (11,-0.12);
\draw [line width=0.5] (12,0) -- (12,-0.12);
\draw [line width=0.5] (13,0) -- (13,-0.12);
\draw [line width=0.5] (14,0) -- (14,-0.12);
\draw [line width=0.5] (15,0) -- (15,-0.12);
\draw [line width=0.5] (-3,1) -- (-3.12,1);
\draw [line width=0.5] (-3,2) -- (-3.12,2);
\draw [line width=0.5] (-3,3) -- (-3.12,3);
\draw [line width=0.5] (-3,4) -- (-3.12,4);
\draw [dotted] (-3,1) -- (15.5,1);
\draw [dotted] (-3,2) -- (15.5,2);
\draw [dotted] (-3,3) -- (15.5,3);
\draw [dotted] (-3,4) -- (15.5,4);
\draw [dotted] (-2,0) -- (-2,4.5);
\draw [dotted] (-1,0) -- (-1,4.5);
\draw [dotted] (0,0) -- (0,4.5);
\draw [dotted] (1,0) -- (1,4.5);
\draw [dotted] (2,0) -- (2,4.5);
\draw [dotted] (3,0) -- (3,4.5);
\draw [dotted] (4,0) -- (4,4.5);
\draw [dotted] (5,0) -- (5,4.5);
\draw [dotted] (6,0) -- (6,4.5);
\draw [dotted] (7,0) -- (7,4.5);
\draw [dotted] (8,0) -- (8,4.5);
\draw [dotted] (9,0) -- (9,4.5);
\draw [dotted] (10,0) -- (10,4.5);
\draw [dotted] (11,0) -- (11,4.5);
\draw [dotted] (12,0) -- (12,4.5);
\draw [dotted] (13,0) -- (13,4.5);
\draw [dotted] (14,0) -- (14,4.5);
\draw [dotted] (15,0) -- (15,4.5);
\draw (-2,-0.1) node[below, scale=0.8]{$1$};
\draw (-1,-0.1) node[below, scale=0.8]{$2$};
\draw (0,-0.1) node[below, scale=0.8]{$3$};
\draw (1,-0.1) node[below, scale=0.8]{$4$};
\draw (2,-0.1) node[below, scale=0.8]{$5$};
\draw (3,-0.1) node[below, scale=0.8]{$6$};
\draw (4,-0.1) node[below, scale=0.8]{$7$};
\draw (5,-0.1) node[below, scale=0.8]{$8$};
\draw (6,-0.1) node[below, scale=0.8]{$9$};
\draw (7,-0.1) node[below, scale=0.8]{$10$};
\draw (8,-0.1) node[below, scale=0.8]{$11$};
\draw (9,-0.1) node[below, scale=0.8]{$12$};
\draw (10,-0.1) node[below, scale=0.8]{$13$};
\draw (11,-0.1) node[below, scale=0.8]{$14$};
\draw (12,-0.1) node[below, scale=0.8]{$15$};
\draw (13,-0.1) node[below, scale=0.8]{$16$};
\draw (14,-0.1) node[below, scale=0.8]{$17$};
\draw (15,-0.1) node[below, scale=0.8]{$18$};
\draw (-3.1,1) node[left, scale=0.8]{$1$};
\draw (-3.1,2) node[left, scale=0.8]{$2$};
\draw (-3.1,3) node[left, scale=0.8]{$3$};
\draw (-3.1,4) node[left, scale=0.8]{$4$};
\end{tikzpicture}
\end{center}
\caption{Example, in the case $q=3, p=4$, of a path in the collection $\mathcal{D}_{[18,0,3]}$ with weight
$a^{(1)}_{0}a^{(2)}_{1}a^{(-2)}_{1}a^{(3)}_{1}a^{(-3)}_{1}a^{(-1)}_{0}a^{(2)}_{0}
a^{(1)}_{2}a^{(0)}_{3}a^{(1)}_{3}
a^{(-4)}_{0}a^{(2)}_{0}a^{(-1)}_{1}a^{(-1)}_{0}a^{(3)}_{0}a^{(-2)}_{1}a^{(1)}_{1}a^{(1)}_{2}$.}
\label{7-1-22-1}
\end{figure}

\begin{figure}
\begin{center}
\begin{tikzpicture}[scale=0.6]
\draw[line width=1.5pt]  (-3,0) -- (15.5,0);
\draw[line width=1.5pt]  (-3,-2.5) -- (-3,3.5);
\draw[line width=0.7pt] (-3,-1) -- (-2,-2) -- (-1,-1) -- (0,2) -- (1,2) -- (2,-2) -- (3,0) -- (4,-1) -- (5,1) -- (6,3) -- (7,2) -- (8,-1) -- (9,-1) -- (10,-2) -- (11,1) -- (12,0) -- (13,-2) -- (14,1) -- (15,2);
\draw [line width=0.5] (-2,0) -- (-2,-0.12);
\draw [line width=0.5] (-1,0) -- (-1,-0.12);
\draw [line width=0.5] (0,0) -- (0,-0.12);
\draw [line width=0.5] (1,0) -- (1,-0.12);
\draw [line width=0.5] (2,0) -- (2,-0.12);
\draw [line width=0.5] (3,0) -- (3,-0.12);
\draw [line width=0.5] (4,0) -- (4,-0.12);
\draw [line width=0.5] (5,0) -- (5,-0.12);
\draw [line width=0.5] (6,0) -- (6,-0.12);
\draw [line width=0.5] (7,0) -- (7,-0.12);
\draw [line width=0.5] (8,0) -- (8,-0.12);
\draw [line width=0.5] (9,0) -- (9,-0.12);
\draw [line width=0.5] (10,0) -- (10,-0.12);
\draw [line width=0.5] (11,0) -- (11,-0.12);
\draw [line width=0.5] (12,0) -- (12,-0.12);
\draw [line width=0.5] (13,0) -- (13,-0.12);
\draw [line width=0.5] (14,0) -- (14,-0.12);
\draw [line width=0.5] (15,0) -- (15,-0.12);
\draw [line width=0.5] (-3,1) -- (-3.12,1);
\draw [line width=0.5] (-3,2) -- (-3.12,2);
\draw [line width=0.5] (-3,3) -- (-3.12,3);
\draw [line width=0.5] (-3,0) -- (-3.12,0);
\draw [line width=0.5] (-3,-1) -- (-3.12,-1);
\draw [line width=0.5] (-3,-2) -- (-3.12,-2);
\draw [dotted] (-3,1) -- (15.5,1);
\draw [dotted] (-3,2) -- (15.5,2);
\draw [dotted] (-3,3) -- (15.5,3);
\draw [dotted] (-3,-1) -- (15.5,-1);
\draw [dotted] (-3,-2) -- (15.5,-2);
\draw [dotted] (-2,-2.5) -- (-2,3.5);
\draw [dotted] (-1,-2.5) -- (-1,3.5);
\draw [dotted] (0,-2.5) -- (0,3.5);
\draw [dotted] (1,-2.5) -- (1,3.5);
\draw [dotted] (2,-2.5) -- (2,3.5);
\draw [dotted] (3,-2.5) -- (3,3.5);
\draw [dotted] (4,-2.5) -- (4,3.5);
\draw [dotted] (5,-2.5) -- (5,3.5);
\draw [dotted] (6,-2.5) -- (6,3.5);
\draw [dotted] (7,-2.5) -- (7,3.5);
\draw [dotted] (8,-2.5) -- (8,3.5);
\draw [dotted] (9,-2.5) -- (9,3.5);
\draw [dotted] (10,-2.5) -- (10,3.5);
\draw [dotted] (11,-2.5) -- (11,3.5);
\draw [dotted] (12,-2.5) -- (12,3.5);
\draw [dotted] (13,-2.5) -- (13,3.5);
\draw [dotted] (14,-2.5) -- (14,3.5);
\draw [dotted] (15,-2.5) -- (15,3.5);
\draw (-2,-0.1) node[below, scale=0.8]{$1$};
\draw (-1,-0.1) node[below, scale=0.8]{$2$};
\draw (0,-0.1) node[below, scale=0.8]{$3$};
\draw (1,-0.1) node[below, scale=0.8]{$4$};
\draw (2,-0.1) node[below, scale=0.8]{$5$};
\draw (3,-0.1) node[below, scale=0.8]{$6$};
\draw (4,-0.1) node[below, scale=0.8]{$7$};
\draw (5,-0.1) node[below, scale=0.8]{$8$};
\draw (6,-0.1) node[below, scale=0.8]{$9$};
\draw (7,-0.1) node[below, scale=0.8]{$10$};
\draw (8,-0.1) node[below, scale=0.8]{$11$};
\draw (9,-0.1) node[below, scale=0.8]{$12$};
\draw (10,-0.1) node[below, scale=0.8]{$13$};
\draw (11,-0.1) node[below, scale=0.8]{$14$};
\draw (12,-0.1) node[below, scale=0.8]{$15$};
\draw (13,-0.1) node[below, scale=0.8]{$16$};
\draw (14,-0.1) node[below, scale=0.8]{$17$};
\draw (15,-0.1) node[below, scale=0.8]{$18$};
\draw (-3.1,1) node[left, scale=0.8]{$1$};
\draw (-3.1,2) node[left, scale=0.8]{$2$};
\draw (-3.1,3) node[left, scale=0.8]{$3$};
\draw (-3.1,0) node[left, scale=0.8]{$0$};
\draw (-3.1,-1) node[left, scale=0.8]{$-1$};
\draw (-3.1,-2) node[left, scale=0.8]{$-2$};
\end{tikzpicture}
\end{center}
\caption{Example, in the case $q=3, p=4 $, of a path in the collection $\mathcal{P}_{[18,-1,2]}$ with weight
$a^{(-1)}_{-2}a^{(1)}_{-2}\,a^{(3)}_{-1}\,a^{(0)}_{2}a^{(-4)}_{-2}a^{(2)}_{-2}\,
a^{(-1)}_{-1}a^{(2)}_{-1}\,a^{(2)}_{1}
a^{(-1)}_{2}a^{(-3)}_{-1}a^{(0)}_{-1}\,a^{(-1)}_{-2}a^{(3)}_{-2}\,a^{(-1)}_{0}a^{(-2)}_{-2}a^{(3)}_{-2}\,a^{(1)}_{1}$.}
\label{7-1-22-2}
\end{figure}

In what follows we discuss first the construction of the matrix continued fraction associated with the collections of paths $\mathcal{D}_{[n,i,j]}$, for more details see Section \ref{MCF}. For each collection $\mathcal{D}_{[n,i,j]}$ we define the corresponding \emph{weight polynomial} $A_{[n,i,j]}$ given by the formula
\begin{equation}\label{def:weightAnij}
A_{[n,i,j]} :=\sum_{\gamma\in\mathcal{D}_{[n,i,j]}}w(\gamma),
\end{equation}
which we use to generate the formal Laurent series
\[
A_{i,j}(z)=\sum_{n=0}^{\infty}\frac{A_{[n,i,j]}}{z^{n+1}},\qquad i,j\geq 0,
\]				
where $A_{[0,i,j]}=1$ if $i=j$ and $A_{[0,i,j]}=0$ if $i\neq j$. Consider the $q\times p$ matrix 
\[
F(z)=\begin{pmatrix}
A_{0,0}(z) &\cdots & A_{0, p-1}(z)\\
\vdots &\ddots& \vdots\\
A_{q-1, 0}(z) & \cdots&A_{q-1, p-1}(z)\\
\end{pmatrix},
\]
with entries $A_{i,j}(z)$, $0\leq i\leq q-1$, $0\leq j\leq p-1$ just defined.
\par
One of the main results of the paper is Theorem \ref{TH6-8-23-1} in Section \ref{MCF}, which establishes the following matrix continued fraction representation for the matrix $F(z)$:
\begin{equation}\label{6-1-23}
F(z)=\cfrac{{\bf 1}}{\alpha_{0}(z)+\alpha^{+}_0\cfrac{{\bf 1}}{\alpha_{1}(z)+\alpha^{+}_1\cfrac{{\bf 1}}{\alpha_{2}(z)+\alpha^{+}_2\cfrac{{\bf 1}}
{\alpha_{3}(z)+\ddots}\,\alpha^{-}_2}\,\alpha^{-}_1}\,\alpha^{-}_0}. 
\end{equation}
We make here several important observations concerning \eqref{6-1-23}. This and all other continued fractions presented in this work should be understood as a composition of certain transformations (see e.g. formula \eqref{finiteMCFforF}) between matrices of size $q\times p$, which in the case of \eqref{6-1-23} are defined by 
\[
\tau_{\alpha,k}(X)=\frac{\mathbf{1}}{\alpha_{k}(z)+\alpha_{k}^{+}\,X\,\alpha_{k}^{-}}.
\]
These transformations use a pseudo-quotient operation $A=\mathbf{1}/B$ defined in \eqref{opinv1}--\eqref{26-4-4}, which is the main ingredient in the construction of our matrix continued fractions.

The matrix coefficients that appear in \eqref{6-1-23} are defined as follows. For each $k\geq 0$, $\alpha^{+}_k$ denotes the $q\times q$  matrix
\begin{equation}\label{6-9-23-10}
\alpha^{+}_k=\begin{pmatrix}
1 & 0 & \cdots & 0 & 0\\
0 & 1 & \cdots & 0 & 0\\
\vdots & \vdots & \ddots & \vdots & \vdots \\
0 & 0 & \cdots & 1 & 0\\
-a^{(1)}_k & -a_{k}^{(2)} & \cdots & -a^{(q-1)}_k & -a_{k}^{(q)}
\end{pmatrix},
\end{equation}
$\alpha^{-}_k$ is the $p\times p$  matrix
\begin{equation}\label{6-9-23-11}
\alpha^{-}_k=\begin{pmatrix}
1 & 0 & \cdots & 0 & a_{k}^{(-1)}\\[0.2em]
0 & 1 & \cdots & 0 & a_{k}^{(-2)}\\
\vdots & \vdots & \ddots & \vdots & \vdots \\
0 & 0 & \cdots & 1 & a_{k}^{(-p+1)}\\[0.1em]
0 & 0 & \cdots & 0 & a_{k}^{(-p)}
\end{pmatrix},
\end{equation}
and $\alpha_{k}(z)$ is the $q\times p$ matrix
\begin{equation}\label{6-9-23-12}
\alpha_{k}(z)=\begin{pmatrix}
0 &\cdots & 0 & 0\\
\vdots &\ddots& \vdots & \vdots\\
0 & \cdots & 0 & 0 \\
0 & \cdots & 0 & z-a^{(0)}_k
\end{pmatrix}.
\end{equation}
To better understand the derivation of the different matrix continued fractions in this paper, it is crucial to observe the location in the matrix $H$ of the entries $a_{k}^{(j)}$ that appear in the matrices $\alpha_{k}(z), \alpha_{k}^{+}, \alpha_{k}^{-}$ (all other continued fractions are constructed similarly). The entries chosen are the $(k+1)$-st entry on the main diagonal of $H$ (for $\alpha_{k}(z)$), the entries from the $(k+1)$-st row located to the right of the main diagonal (for $\alpha_{k}^{+}$), and the entries from the $(k+1)$-st column located below the main diagonal (for $\alpha_{k}^{-}$).
\par
Theorem \ref{TH6-8-23-1} gives a combinatorial interpretation for a class of matrix continued fractions in terms of lattice paths, in the spirit of the previous work by Flajolet \cite{Fla} and Viennot \cite{Viennot} on combinatorial properties of continued fractions in the scalar case $p=q=1$. Previously, L\'{o}pez-Garc\'{i}a and Prokhorov \cite{LopPro3} studied the case $q=1$, $p\geq 1$, giving a combinatorial interpretation for a class of vector continued fractions, and studying related spectral properties of banded Hessenberg operators. 
\par
The key element in the proof of  \eqref{6-1-23} is Theorem \ref{theorem1} in Section \ref{Lattice paths}, which establishes algebraic relations between the formal Laurent series $A_{i,j}(z)$ associated with the collections $\mathcal{D}_{[n,i,j]}$ of paths restricted to the plane $y\geq 0$, and the Laurent series $A_{i,j}^{(1)}(z)$ associated with collections $\mathcal{D}^{(1)}_{[n,i,j]}\subset\mathcal{D}_{[n,i+1,j+1]}$ of paths restricted to the plane $y\geq 1$. The paths in the collection $\mathcal{D}^{(1)}_{[n,i,j]}$, $i,j \geq 0$, can be obtained by shifting the paths in the collection $\mathcal{D}_{[n,i,j]}$ one unit upwards.
We present two separate proofs of Theorem \ref{theorem1}. In Section \ref{Lattice paths}, the proof relies solely on combinatorial arguments, and in Section \ref{sec:rfH} we give an alternative proof (see Theorem \ref{theorem3}) that uses an approach based on the resolvents of the banded matrix $H$.
\par
Another essential contribution of the paper is Theorem \ref{TH6-9-23-1} in Section \ref{MCF}, which establishes the continued fraction representation for a $q\times p$ matrix of formal Laurent series associated with the collections $\{\mathcal{P}_{[n,i,j]}\}_{n\geq 0}$ of unrestricted lattice paths. We briefly discuss this result below. Recall that $\mathcal{P}_{[n,i,j]}$ consists of paths of length $n$ that start at the point $(0,i)$, end at the point $(n,j)$, and are allowed to go below the $x$-axis. The weight polynomial associated with the collection $\mathcal{P}_{[n,i,j]}$ is denoted $W_{[n,i,j]}$ and is defined similarly to \eqref{def:weightAnij} by the formula
\[
W_{[n,i,j]}:=\sum_{\gamma\in\mathcal{P}_{[n,i, j]}}w(\gamma),
\] 
where $w(\gamma)$ denotes the weight of a path $\gamma$. For integers $i, j$ we define the formal Laurent series $W_{i,j}(z)$ generated by the weight polynomials $W_{[n,i,j]}$:
\[
W_{i,j}(z)  :=\sum_{n=0}^{\infty}\frac{W_{[n,i,j]}}{z^{n+1}},
\]
where $W_{[0,i,j]}=1$ if $i=j$ and $W_{[0,i,j]}=0$ if $i\neq j$. Consider the $q\times p$ matrix
\[
G(z)=\begin{pmatrix}
W_{0,0}(z) &\cdots & W_{0, p-1}(z)\\
\vdots &\ddots& \vdots\\
W_{q-1, 0}(z) & \cdots&W_{q-1, p-1}(z)\\
\end{pmatrix},
\]
with entries $W_{i,j}(z)$, $0\leq i\leq q-1$, $0\leq j\leq p-1$. According to Theorem \ref{TH6-9-23-1}, the matrix $G(z)$ can be expressed as a matrix continued fraction in the following form:
\begin{equation}\label{6-6-23-1}
G(z)=\cfrac{{\bf 1}}{\beta_{0}(z)+\beta^{+}_0\cfrac{{\bf 1}}{\beta_{1}(z)+\beta^{+}_1\cfrac{{\bf 1}}{\beta_{2}(z)+\beta^{+}_2\cfrac{{\bf 1}}
{\beta_{3}(z)+\ddots}\,\beta^{-}_2}\,\beta^{-}_1}\,\beta^{-}_0}, 
\end{equation}
where the matrix coefficients in \eqref{6-6-23-1} are defined in \eqref{6-9-23-40}--\eqref{6-9-23-53}. The matrices used in this continued fraction are constructed by extracting entries from a banded matrix $K=(K_{i,j})_{ i,j=0}^{\infty}$, which has the following entries
\begin{equation}\label{defKintro}
\begin{aligned}
K_{i,j} &=h_{i,j}+\sum ^{p-i}_{l=1}\sum^{q-j}_{m=1}a^{(-(l+i))}_{-l}a^{(m+j)}_{-m}V_{-l,-m}(z), \qquad  0\le i \le p-1,\quad  0 \le j \le
 q-1,\\
K_{i,j} &=h_{i,j}, \qquad \mbox{otherwise},
\end{aligned}
\end{equation}
where $h_{i,j}$ is the $(i,j)$-entry of $H$. These entries in the banded matrix $K$ are identical to those in the matrix $H$ except for the $p \times q$ submatrix located in the top left corner. The entries in this special submatrix are expressed in terms of the entries of the matrix $W$ and the Laurent series $V_{i,j}(z)$ associated with collections of lattice paths that \textit{never go above the line $y=-1$}. For detailed explanations, including precise formulas and proofs regarding the matrix continued fraction \eqref{6-6-23-1}, see Sections \ref{Wmatrix} and \ref{MCF}.
\par
The proof of \eqref{6-6-23-1} relies on Theorem \ref{10-3-22-1} in Section \ref{Wmatrix}, which provides a crucial tool. This theorem enables us to simplify the task of finding the matrix continued fraction for $G(z)$ by reducing it to finding the matrix continued fraction constructed from the matrix $K$.
\par
It is important to emphasize that in both Theorems \ref{10-3-22-1} and \ref{theorem3}, we used Theorem \ref{theorem2} in Section \ref{sec:lattpathoper}, which shows that the Laurent series $A_{i,j}(z)$ and $W_{i,j}(z)$ can be characterized as  resolvent functions associated with the banded matrices $H$ and $W$, respectively. 
\par
To the best of our knowledge, \eqref{6-6-23-1} is the first matrix continued fraction expansion in the literature for resolvents of two-sided difference operators of an arbitrary finite order. The interesting form of a double continued fraction that \eqref{6-6-23-1} takes in the scalar case $p=q=1$ is discussed in subsection \ref{subsecscf}, see \eqref{scalardcf}. The scalar case $p=q=1$ is of course well researched; relatively recent important works are \cite{DIW,DKS,MasRep,Nik,Simon}. Observe that the expansion \eqref{6-1-23} is valid for arbitrary one-sided difference operators of finite order.

Matrix continued fractions have been investigated extensively by Sorokin and Van Iseghem \cite{SorVanIseg1,SorVanIseg2,SorVanIseg3,VanIseg2,VanIseg4}, in connection with closely related subjects such as matrix Hermite-Pad\'{e} approximation, vector and matrix orthogonality, vector recurrence relations, and discrete dynamical systems. This paper extends their work in several respects. First, we have obtained our results in the context of general one-sided and two-sided difference operators of arbitrary finite order. Second, as it is discussed below, we have estimated the degree of approximation to the resolvents of $H$ by the resolvents of the principal $n\times n$ truncations $H_{n}=(h_{i,j})_{i,j=0}^{n-1}$ of $H$, of which we have explicit formulas, see \eqref{rationalapprox1}--\eqref{rationalapprox5}. Third, we show that the matrix continued fractions of the approximating resolvents of $H_{n}$ can be regarded as special convergents of the matrix continued fraction for $H$.  

In the vector case $q=1$, $p\geq 1$, some important recent works on the algebraic and analytic aspects of vector continued fractions are \cite{AptKal,AptKalVan,AptKalSaff,Kal,Kal2,NikSor,VanIseg3,VanIseg4}, see also \cite{LopPro3}.

\subsection{Random banded matrices and their characteristic  polynomials}

The investigation of the distribution of eigenvalues of random banded matrices is another essential part of the paper. We consider a class of random banded matrices $H$, as defined in \eqref{matrixH}, with $p+q+1$ diagonals, where the entries are independent and bounded random variables. We assume that entries in the same diagonal have identical distributions, but the distributions may be different for different diagonals. 
\par
For the principal $n\times n$ truncation $H_{n}=(h_{i,j})_{i,j=0}^{n-1}$ of the banded matrix $H$, we analyze the asymptotic behavior of its empirical spectral distribution as $n$ approaches infinity. Theorem \ref{10-26-2022} in Section \ref{randompol} specifically addresses the asymptotic behavior of the expected values of the eigenvalue moments for $H_{n}$ and its connection with the Laurent series $W_{0,0}(z)$.
\par
This result extends the work of  L\'{o}pez-Garc\'{i}a and Prokhorov \cite{LopPro2}, which was  focused on the case $q=1$, $p\geq 1$. Certain combinatorial aspects of the theory of random characteristic polynomials in the case $p=q=1$ were investigated in \cite{LopPro}.

\subsection{Resolvent functions of the principal $n\times n$ truncation of the banded matrix $H$ and rational approximation}

In Section \ref{12-12-22-1} we investigate the resolvent functions associated with the principal $n\times n$ truncation $H_{n}$ of the banded matrix $H$ and their relationship to rational approximation. We prove that the $q\times p$ matrix of resolvents associated with $H_n$  approximates the matrix $F(z)$ of resolvents for $H$. Furthermore, we express the matrix of resolvents for $H_{n}$ as a finite matrix continued fraction (see Proposition \ref{prop:MCFforRn}), and show that it can be regarded as a convergent for the matrix continued fraction for $F(z)$.

Let $Q_{n}(z)=\det(z I_{n}-H_{n})$ denote the characteristic polynomial of the matrix 
$H_{n}$, which is a monic polynomial of degree $n$.
We define the resolvent functions as
\begin{equation}\label{rationalapprox1}
R_{i,j,n}(z)=\langle (zI_n-H_{n})^{-1} e_{j}, e_{i}\rangle,
\end{equation}
where $\{e_{i}\}_{i=0}^{n-1}$
represents the standard basis in $\mathbb{C}^{n}$. These resolvent functions $R_{i,j,n}(z)$, $0\leq i,j \leq n-1$, are rational functions and can be expressed as 
\begin{equation}\label{rationalapprox2}
R_{i,j,n}(z)=\frac{P_{i,j,n}(z)}{Q_n(z)},
\end{equation}
where the numerators $P_{i,j,n}(z)$ are polynomials of degree at most $n-1$, given by the expression 
\begin{equation}\label{rationalapprox3}
P_{i,j,n}(z) =(-1)^{i+j}\det((z I_n-H_{n})^{[j,i]}).
\end{equation}
Here, $(zI_n-H_{n})^{[j,i]}$ denotes the submatrix of $zI_n-H_{n}$ obtained by removing the $j$-th row and the $i$-th column.
\par 
In Section \ref{12-12-22-1} we show that the rational functions $R_{i,j,n}(z)$ can be used as rational approximants for the formal power series $A_{i,j}(z)$ in the range $0\leq i,j \leq n-1$. More precisely, the coefficients in the Laurent series expansions of $A_{i,j}(z)$ and $R_{i,j,n}(z)$ must match up to a certain order. As $n$ approaches infinity, this order increases to infinity like $n(1/p + 1/q)$. In more detail, for $0\leq i,j\leq n-1$ we have
\begin{equation}\label{rationalapprox4}
A_{i,j}(z)-R_{i,j,n}(z)=O(z^{-L-2}), \qquad z\rightarrow \infty,
\end{equation}
where 
\begin{equation}\label{rationalapprox5}
L=\left[\frac{n-1-i}{q}\right] + \left[\frac{n-1-j}{p}\right]+1,
\end{equation}
and $[\cdot]$ represents the integer part. 
\par
Let $R_n(z)$ denote the $q\times p$ matrix
\begin{equation}\label{defRnintro}
R_n(z):=\begin{pmatrix}
R_{0,0, n}(z) &\cdots & R_{0, p-1, n}(z)\\
\vdots &\ddots& \vdots\\
R_{q-1, 0, n}(z) & \cdots & R_{q-1, p-1, n}(z)\\
\end{pmatrix},
\end{equation}
where the $(i,j)$-entry is given by the rational function $R_{i,j,n}(z)$. In this paper we show that the matrix $R_n(z)$ is a special type of convergent for the matrix continued fraction expansion \eqref{6-1-23} of the matrix $F(z)$ (for more details, see the discussion after the proof of Theorem \ref{5-16-23-11} in Section \ref{12-12-22-1}).
\par
It is important to emphasize that in the case $p=q=1$, the rational function $R_{0,0,n}(z)$ is the classical Pad\'{e} approximant of the Laurent series $A_{0,0}(z)$. Additionally, when $q=1$ and $p\geq 1$, the vector $(R_{0,0, n}(z)\,\,\ldots\,\,R_{0, p-1, n}(z))$ is a vector of Hermite-Pad\'{e} approximants (of type II) for the vector of formal series $(A_{0,0}(z)\,\,\ldots\,\,A_{0, p-1}(z)).$ Hermite-Pad\'{e} approximation to systems of resolvent functions of difference operators have been investigated extensively in recent decades, with notable works by Kalyagin \cite{Kal,Kal2}, Aptekarev--Kaliaguine \cite{AptKal}, Aptekarev--Kaliaguine--Van Iseghem \cite{AptKalVan}, Van Iseghem \cite{VanIseg,VanIseg3,VanIseg4}, see also \cite{LopPro2,LopPro3,NikSor,RobSan}. In \cite{SorVanIseg1,SorVanIseg2,SorVanIseg3,VanIseg2,VanIseg4}, Sorokin and Van Iseghem have studied different matrix Hermite-Pad\'{e} problems.

\subsection{Organization of the paper}

This paper is organized as follows. In Section \ref{Lattice paths}, we introduce the precise definitions and notations related to lattice paths, weight polynomials, and spaces of formal Laurent series with complex coefficients, and we prove the key algebraic relations in Theorem~\ref{theorem1} between the formal series $A_{i,j}(z)$ and $A_{i,j}^{(1)}(z)$ using combinatorial arguments. In Section \ref{sec:lattpathoper} we prove Theorem \ref{theorem2}, which characterizes the generating series $A_{i,j}(z)$ and $W_{i,j}(z)$ as resolvent functions of one-sided and two-sided operators, respectively. In Section 4 we introduce formal series with operator coefficients, and give in Theorem \ref{theorem3} an alternative proof of Theorem \ref{theorem1}, based on matrix partitioning and formulas for the inverse of partitioned matrices. Section \ref{Wmatrix} is devoted to the analysis of the resolvent functions of the matrix $W$ in \eqref{matrixW} and their relation to the matrix $K=(K_{i,j})_{i,j=0}^{\infty}$ in \eqref{defKintro}, which is the crucial ingredient for the construction of the matrix continued fraction \eqref{6-6-23-1}. In Section \ref{MCF} we deduce the matrix continued fractions \eqref{6-1-23} and \eqref{6-6-23-1} in our main Theorems \ref{TH6-8-23-1} and \ref{TH6-9-23-1}. Section \ref{12-12-22-1} discusses the approximation property of the resolvents of the principal $n\times n$ truncations $H_{n}$ to the resolvents of $H$ using some elementary combinatorial arguments based on lattice paths, as well as the matrix continued fraction for the matrix \eqref{defRnintro}. Finally, in the last section we prove our result on characteristic polynomials of random banded matrices. 

\section{Lattice paths and generating series for weight polynomials}\label{Lattice paths}
Fix integers $p,q \geq 1$ throughout the paper. Denote by $\mathcal{G}=(V,E)$ the oriented graph with set of vertices $V:=\mathbb{Z}_{\geq 0}\times \mathbb{Z}$ and set $E$ of edges (steps) of the following form:
\begin{equation}\label{defsteps}
\begin{aligned}
& \mbox{upsteps}\,\,\,(n,m)\rightarrow(n+1,m+i)\,\,\,\mbox{by $i$ units},\,\,1 \leq i \leq q, \\
& \mbox{level steps}\,\,\,(n,m)\rightarrow(n+1,m), \,\,\,\, \\
& \mbox{downsteps}\,\,\,(n,m)\rightarrow(n+1,m-j)\,\,\,\mbox{by $j$ units},\,\,1\leq j \leq p,\\
\end{aligned}
\end{equation}
where the notation $v\rightarrow v'$ indicates the step from vertex $v$ to vertex $v'$. The difference in height between the vertices in an upstep is a value $i\in\{1,\ldots,q\}$, and
the difference in height between the vertices in a downstep is a value $j\in\{1,\ldots,p\}.$
A \emph{lattice path} on $\mathcal{G}$  of length $k$
is a  sequence of  $k$ steps
\[
\gamma=e_{1}e_{2}\cdots e_{k},
\]
where for each $1\leq j\leq k-1,$ the final vertex of $e_{j}$ conicides with the initial vertex of $e_{j+1}$.  A path of length zero is simply a vertex in $V$. If $(n,m)\in V$ is a vertex in the path $\gamma$, we say that $\gamma$ has height $m$ at $n$.

As above, for each $-p\leq k\leq q$, let $(a^{(k)}_{n})_{n\in\mathbb{Z}}$ be a bi-infinite sequence of complex numbers.
To each step we give a \emph{weight} as follows:
\begin{equation}\label{6-29-22-2}
\begin{aligned}
w((n,m)\rightarrow (n+1,m+i)) & =a_{m}^{(i)},\qquad 0\leq i\leq q,\\
w((n,m)\rightarrow (n+1,m-j)) & =a_{m-j}^{(-j)},\qquad 1\leq j\leq p.
\end{aligned}
\end{equation}
The weight of a path $\gamma$ is defined by
\[
w(\gamma)=\prod_{e\subset\gamma}w(e),
\]
where the product runs over the different steps of $\gamma$. The weight of a path of length zero is by definition $1$.

If $\mathcal{S}$ is a finite collection of lattice paths, the expression $\sum_{\gamma\in\mathcal{S}}w(\gamma)$ will be called the \textit{weight polynomial associated with $\mathcal{S}$}. If $\mathcal{S}$ is the empty collection, its weight polynomial is understood to be zero. 

For each $n\geq 0$ and $i,j \in \mathbb{Z}$, we denote by $\mathcal{P}_{[n, i, j]}$ the collection of all paths of length $n$, with initial
point $(0,i)$ and final point $(n,j)$. Let $\mathcal{D}_{[n,i,j]}$, $i,j \geq 0$, be the subcollection of $\mathcal{P}_{[n,i,j]}$ consisting of those paths with no point below the $x$-axis.  

The weight polynomials associated with the collections $\mathcal{D}_{[n,i,j]}$ and $\mathcal{P}_{[n,i, j]}$ are denoted
\begin{align*}
A_{[n,i,j]} & :=\sum_{\gamma\in\mathcal{D}_{[n,i,j]}}w(\gamma)\qquad n \geq 0, \quad i,j \geq 0,\\
W_{[n,i,j]} & :=\sum_{\gamma\in\mathcal{P}_{[n,i, j]}}w(\gamma)\qquad n \geq 0, \quad i,j \in \mathbb{Z}.
\end{align*} 
In our analysis we will also need one more collection of lattice paths. 
Let $$\mathcal{D}^{(1)}_{[n,i,j]}, \ \ n \geq 0,\quad i,j \geq 0,$$ be the subcollection of $\mathcal{D}_{[n,i+1,j+1]}$ consisting of those paths with no point below the line $y=1$. The paths in the collection $\mathcal{D}^{(1)}_{[n,i,j]}$ can be obtained by shifting the paths in the collection $\mathcal{D}_{[n,i,j]}$ $1$ unit upwards. We denote by $\gamma+1$ the path resulting from this operation applied to $\gamma$. So we have 
$$
\mathcal{D}^{(1)}_{[n,i,j]}:=\{\gamma+1: \gamma\in \mathcal{D}_{[n,i,j]}\}.
$$
The weight polynomial associated with $\mathcal{D}^{(1)}_{[n,i,j]}$ is 
\[
A_{[n,i,j]}^{(1)} :=\sum_{\gamma\in\mathcal{D}^{(1)}_{[n,i,j]}}w(\gamma)=\sum_{\gamma\in\mathcal{D}_{[n,i,j]}}w(\gamma+1)
\qquad n \geq 0,\quad i,j \geq 0.
\]
Note that the formula for $A_{[n,i,j]}^{(1)}$ is obtained from the formula for $A_{[n,i,j]}$ by replacing each term $a_{m}^{(k)}$ by $a_{m+1}^{(k)}$ in the expression of $A_{[n,i,j]}$.

We introduce now certain formal Laurent series generated by the sequences of weight polynomials defined. These series are considered as elements of the algebraic field $\mathbb{C}((z^{-1}))$ (see e.g. \cite[Chapter 2]{NikSor}) consisting of all formal power series
\[
a(z)=\sum_{n\in\mathbb{Z}}a_{n} z^{n}
\]
with complex coefficients such that only finitely many $a_{n}$ with $n>0$ are non-zero. The sum and product of two series $a(z)=\sum_{n\in\mathbb{Z}} a_{n} z^{n}$ and $b(z)=\sum_{n\in\mathbb{Z}} b_{n} z^{n}$ are defined by
\begin{align*}
(a+b)(z) & =\sum_{n\in\mathbb{Z}}(a_{n}+b_{n})\,z^{n},\\
(a\cdot b)(z) & =\sum_{n\in\mathbb{Z}}\Big(\sum_{l\in\mathbb{Z}} a_{l}\,b_{n-l}\Big) z^{n}.
\end{align*}
Observe that in the second formula the inner sum is finite. The degree of a series $a(z)=\sum_{n\in\mathbb{Z}} a_{n} z^{n}$ is by definition $\deg(a(z))=\max\{n\in\mathbb{Z}: a_{n}\neq 0\}$. We will use the standard notation $[z^n]\,a(z)$ to indicate the coefficient of $z^{n}$ in the series $a(z)$.

For nonnegative integers $i, j$, let
\begin{align}
A_{i,j}(z) & :=\sum_{n=0}^{\infty}\frac{A_{[n,i,j]}}{z^{n+1}}, \label{7-1-22-6}\\
A_{i,j}^{(1)}(z) & :=\sum_{n=0}^{\infty}\frac{A_{[n,i,j]}^{(1)}}{z^{n+1}}, \label{7-1-22-7}
\end{align}
and for integers $i, j$, we define
\begin{equation}\label{7-1-22-8}
W_{i,j}(z)  :=\sum_{n=0}^{\infty}\frac{W_{[n,i,j]}}{z^{n+1}}.
\end{equation}

Suppose that $i\leq j$, and let $m$ be the smallest integer such that $mq\geq j-i$. Then for $0\leq r<m$ we have $rq<j-i$, and so $\mathcal{D}_{[r,i,j]}=\emptyset$ and $A_{[r,i,j]}=0$. Therefore $A_{i,j}(z)=O(z^{-m-1})$. Similarly, if $i>j$ and $s$ is the smallest integer such that $sp\geq i-j$, then it is easy to see that $A_{i,j}(z)=O(z^{-s-1})$. The same estimates are valid for the series \eqref{7-1-22-7} and \eqref{7-1-22-8}. 

\begin{theorem}\label{theorem1}
The following relations hold between the series defined in \eqref{7-1-22-6} and \eqref{7-1-22-7}:
\begin{align}
A_{0,0}(z) & =\frac{1}{z-a_{0}^{(0)}-\sum_{i=1}^{q}\sum_{j=1}^{p}a_{0}^{(i)}a_{0}^{(-j)}\,A_{i-1,j-1}^{(1)}(z)}, \label{7-2-22-1}\\
A_{0,j}(z) & =A_{0,0}(z)\sum_{i=1}^{q} a_{0}^{(i)}A_{i-1,j-1}^{(1)}(z),  \qquad j\geq 1,\label{7-2-22-2}\\
A_{i,0}(z) & =A_{0,0}(z)\sum_{j=1}^{p} a_{0}^{(-j)}A_{i-1,j-1}^{(1)}(z), \qquad i\geq 1, \label{7-2-22-3}\\
A_{i,j}(z) & =\frac{A_{i,0}(z)\,A_{0,j}(z)}{A_{0,0}(z)}+A^{(1)}_{i-1,j-1}(z), \qquad i, j\geq 1.
\label{6-22-22}
\end{align}
\end{theorem}
\begin{proof}
First we prove \eqref{7-2-22-2}, so let $j\geq 1$ be fixed. If $n$ is a nonnegative integer and $qn \leq j-1$, then $\mathcal{D}_{[n,0,j]}=\emptyset$ and $A_{[n,0,j]}=0$, because a lattice path starting at height $0$ cannot reach height $j$ in $n$ steps. So assume $qn\geq j$ and let $\gamma$ be a path in $\mathcal{D}_{[n,0,j]}$. We can find in $\gamma$ a unique upstep with the following property: it is the last upstep of the form $(\kappa,0)\rightarrow (\kappa+1,i)$,  $1\leq i \leq q$, as we cross the path from left to right. In other words, it is the last upstep that starts at height $0$. This step clearly exists since the path starts at height $0$ and ends at height $j\geq 1$. We denote the abscissa of the initial point of this unique upstep with the symbol $\kappa(\gamma)$.
\par
 The path $\gamma\in\mathcal{D}_{[n,0,j]}$ can be subdivided into three parts $\gamma_{1}$, $\gamma_{2}$, $\gamma_{3}$. The piece $\gamma_{1}$ is the portion of $\gamma$ on the interval $[0,\kappa(\gamma)]$. The piece $\gamma_{2}$ is the unique upstep mentioned above $(\kappa(\gamma),0)\rightarrow(\kappa(\gamma)+1,i)$,  $1\leq i \leq q$, and $\gamma_{3}$ is the portion of $\gamma$ on the interval 
$[\kappa(\gamma)+1, n]$. We note that 
\begin{equation}\label{7-3-22-2}
w(\gamma) = w(\gamma_{1}) \cdot w(\gamma_{2}) \cdot w(\gamma_{3}).
\end{equation}
The first piece $\gamma_{1}$ begins at $(0,0)$ and ends at $(\kappa(\gamma),0)$, so it is a path in $\mathcal{D}_{[\kappa(\gamma),0,0]}$ since $\gamma_{1}$ has length $\kappa(\gamma)$ and starts and ends at height 0. The second piece $\gamma_{2}$ is an upstep, so by \eqref{6-29-22-2} we have $w(\gamma_{2})=a_{0}^{(i)}$. The third piece $\gamma_{3}$ begins at $(\kappa(\gamma)+1,i)$ and ends at $(n,j)$. Since $\gamma_{3}$ has no point below the line $y=1$, $\gamma_{3}$ can be identified with a horizontal translation of a path in $\mathcal{D}^{(1)}_{[n-\kappa(\gamma)-1,i-1,j-1]}$.

Let us now write the formula for the weight polynomial $A_{[n,0,j]}$. 
Taking into account \eqref{7-3-22-2}, the fact that $i$ varies from $1$ to $q$, and the possible values for $\kappa(\gamma)$ between $0$ and $n-1$, we conclude that 
\[
A_{[n,0,j]}=\sum_{\gamma\in\mathcal{D}_{[n,0,j]}}w(\gamma)=\sum_{\kappa=0}^{n-1}\sum_{i=1}^{q}A_{[\kappa,0,0]}\,a_{0}^{(i)}\,A_{[n-\kappa-1,i-1,j-1]}^{(1)},\qquad n\geq 1.
\]
Note also that $A_{[0,0,j]}=0$ since $j\geq 1$. Consequently, we deduce 
\begin{align*}
A_{0,j}(z) &=\sum_{n=1}^{\infty}\frac{A_{[n,0,j]}}{z^{n+1}}=\sum_{n=1}^{\infty}\frac{1}{z^{n+1}}\sum_{\kappa=0}^{n-1}\sum_{i=1}^{q}A_{[\kappa,0,0]}\,a_{0}^{(i)}\,A_{[n-\kappa-1,i-1,j-1]}^{(1)} \\ &=A_{0,0}(z)\sum_{i=1}^{q} a_{0}^{(i)}A_{i-1,j-1}^{(1)}(z).
\end{align*}
So \eqref{7-2-22-2} is proved.
\par
We can prove \eqref{7-2-22-3} in a similar fashion. Let $i\geq 1$ be fixed. If $n\geq 0$ satisfies $pn \leq i-1$, then 
$\mathcal{D}_{[n,i,0]}=\emptyset$ and $A_{[n,i,0]}=0$ because a path starting at height $i$ cannot reach height $0$ in $n$ steps. So assume $pn\geq i$ and let $\gamma$ be a path in $\mathcal{D}_{[n,i,0]}$. We can find in $\gamma$ a unique downstep with the following property: it is the first downstep of the form $(\kappa,j)\rightarrow (\kappa+1,0),$  $1\leq j \leq p$, as we cross the path from left to right. In other words, it is the first downstep that ends at height $0$. This step certainly exists since the path starts at height $i\geq 1$ and ends at height $0$. We denote the abscissa of the initial point of this unique downstep with the symbol $\kappa(\gamma)$.

The path  $\gamma\in\mathcal{D}_{[n,i,0]}$ can be subdivided into three parts $\gamma_{1}$, $\gamma_{2}$, $\gamma_{3}$. The first piece $\gamma_{1}$ is the portion of $\gamma$ on the interval $[0,\kappa(\gamma)]$, the piece $\gamma_{2}$ is the unique downstep mentioned above $(\kappa(\gamma),j)\rightarrow(\kappa(\gamma)+1,0)$, and $\gamma_{3}$ is the portion of $\gamma$ on the interval 
$[\kappa(\gamma)+1, n]$.
\par
The first piece $\gamma_{1}$ begins at $(0,i)$ and ends at $(\kappa(\gamma),j)$, so it  is clearly a path in $\mathcal{D}^{(1)}_{[\kappa(\gamma),i-1,j-1]}$ since  $\gamma_{1}$ has no point below the line $y=1$. The downstep $\gamma_{2}$ has weight $a_{0}^{(-j)}$, cf. \eqref{6-29-22-2}. The third piece $\gamma_{3}$ begins at $(\kappa(\gamma)+1,0)$ and ends at $(n,0)$. Therefore $\gamma_{3}$ can be identified with a horizontal translation of a  path in $\mathcal{D}_{[n-\kappa(\gamma)-1,0,0]}$ since it has length $n-\kappa(\gamma)-1$ and begins and ends at height $0$. 
 
Let us write the formula for the weight polynomial $A_{[n,i,0]}$ associated to the collection $\mathcal{D}_{[n,i,0]}.$
Taking into account the relation $w(\gamma) = w(\gamma_{1}) \cdot w(\gamma_{2}) \cdot w(\gamma_{3})$, the range $1\leq j\leq p$, and the possible values for $\kappa(\gamma)$ between $0$ and $n-1$, we conclude that 
\begin{equation}\label{newrel}
A_{[n,i,0]}=\sum_{\gamma\in\mathcal{D}_{[n,i,0]}}w(\gamma)=
\sum_{\kappa=0}^{n-1}\sum_{j=1}^{p}A_{[n-\kappa-1,0,0]}\,a_{0}^{(-j)}\,A_{[\kappa,i-1,j-1]}^{(1)},\qquad n\geq 1.
\end{equation}
Observe that $A_{[0,i,0]}=0$. We can rewrite the relation \eqref{newrel} in terms of the series defined in \eqref{7-1-22-6} and \eqref{7-1-22-7} and obtain \eqref{7-2-22-3}.
 
Now we prove \eqref{6-22-22}. Let $i, j\geq 1$. If $n\geq 0$ satisfies $qn \leq j-i-1$ or $pn \leq i-j-1$, then $\mathcal{D}_{[n,i,j]}=\emptyset$ and $A_{[n,i,j]}=0$ because a path with initial height $i$ cannot reach height $j$ in $n$ steps. Assume then that $pn \geq i-j$ and  $qn \geq j-i$, and let $\gamma$ be a path in $\mathcal{D}_{[n,i,j]}$.

We consider two cases for $\gamma$. In the first case $\gamma$ does not go below the line $y=1$ at any point. Then $\gamma$ is clearly a path in $\mathcal{D}^{(1)}_{[n,i-1,j-1]}$. If $\gamma$ goes below the line $y=1$ at some point, then there is a unique downstep in $\gamma$ with the following property: it is the first downstep of the form $(\kappa,t)\rightarrow (\kappa+1,0),$  $1\leq t \leq p,$ as we cross the path from left to right. We denote the abscissa of the initial point of this unique downstep with the symbol $\kappa(\gamma)$.

We subdivide $\gamma$ into three pieces $\gamma_{1}$, $\gamma_{2}$, $\gamma_{3}$. The first piece $\gamma_{1}$ starts at $(0,i)$ and ends at $(\kappa(\gamma),t)$.  The path $\gamma_{1}$ has no point below the line $y=1$, so $\gamma_{1}$ belongs to the collection $\mathcal{D}^{(1)}_{[\kappa(\gamma),i-1,t-1]}$. The second piece $\gamma_{2}$ is the downstep $(\kappa(\gamma),t)\rightarrow(\kappa(\gamma)+1,0)$, with weight $a_{0}^{(-t)}$. The third piece $\gamma_{3}$ starts at $(\kappa(\gamma)+1,0)$ and ends at $(n,j)$, so we can identify it with a horizontal translation of a path in $\mathcal{D}_{[n-\kappa(\gamma)-1,0,j]}$. Taking into account $w(\gamma) = w(\gamma_{1}) \cdot w(\gamma_{2}) \cdot w(\gamma_{3})$, the fact that $t$ varies from $1$ to $p$, and $\kappa(\gamma)$ varies between $0$ and $n-1$, we conclude that 
\[
A_{[n,i,j]}=\sum_{\gamma\in\mathcal{D}_{[n,i,j]}}w(\gamma)=
\sum_{\kappa=0}^{n-1}\sum_{t=1}^{p}A_{[n-\kappa-1,0,j]}\,a_{0}^{(-t)}\,A_{[\kappa,i-1,t-1]}^{(1)}+ A_{[n,i-1,j-1]}^{(1)}.
\]
In terms of the Laurent series defined in \eqref{7-1-22-6} and \eqref{7-1-22-7}, the previous relation implies
\[
A_{i,j}(z) =A_{0,j}(z)  \sum_{t=1}^{p} a_{0}^{(-t)}A_{i-1,t-1}^{(1)}(z)+ A_{i-1,j-1}^{(1)}(z).
\]
From this relation and \eqref{7-2-22-3} we obtain \eqref{6-22-22}.

Note that \eqref{7-2-22-1} is equivalent to
\begin{equation}\label{7-4-22-1}
zA_{0,0}(z)-1  = A_{0,0}(z)(a_{0}^{(0)}+\sum_{i=1}^{q} \sum_{j=1}^{p} a_{0}^{(i)} a_{0}^{(-j)} A_{i-1,j-1}^{(1)}(z)),
\end{equation}
se we prove \eqref{7-4-22-1}. Let $n\geq 1$. If $\gamma\in\mathcal{D}_{[n,0,0]}$, then  the first step of $\gamma$ is either the level step $(0,0) \rightarrow (1,0)$ or an upstep $(0,0) \rightarrow (1, i)$ for some  $1\leq i\leq q$. Partitioning the collection of paths $\mathcal{D}_{[n,0,0]}$ according to the first step of a path $\gamma\in\mathcal{D}_{[n,0,0]}$, we can write  the weight polynomial $A_{[n,0,0]}$ as
\begin{equation}\label{7-4-22-5}
A_{[n,0,0]}=\sum_{\gamma\in\mathcal{D}_{[n,0,0]}} w(\gamma)= a_{0}^{(0)}A_{[n-1,0,0]}+
\sum_{i=1}^{q}  a_{0}^{(i)}A_{[n-1,i,0]}.
\end{equation}
Then by \eqref{7-4-22-5} and \eqref{7-2-22-3}, we obtain
\begin{align*}
z A_{0,0}(z)-1 & = a_{0}^{(0)} A_{0,0}(z)+\sum_{i=1}^{q} a_{0}^{(i)}  A_{i,0}(z)=a_{0}^{(0)} A_{0,0}(z)+\sum_{i=1}^{q} \sum_{j=1}^{p} a_{0}^{(i)} a_{0}^{(-j)} A_{0,0}(z) A_{i-1,j-1}^{(1)}(z)\notag\\
 & =A_{0,0}(z) (a_{0}^{(0)}+\sum_{i=1}^{q} \sum_{j=1}^{p} a_{0}^{(i)} a_{0}^{(-j)} A_{i-1,j-1}^{(1)}(z))
\end{align*}
and \eqref{7-4-22-1} follows.\end{proof}

In the following result we gather some linear relations between the series \eqref{7-1-22-6}. 

\begin{proposition}\label{pr5-18-23}
The following relations hold for any $i, j\geq 0$:
\begin{align}
z A_{i,j}(z)-\delta_{i,j} & =\sum_{r=1}^{p} a_{j}^{(-r)}A_{i,j+r}(z)+\sum_{r=0}^{\min\{j,q\}}a_{j-r}^{(r)}A_{i,j-r}(z),\label{newrel2}\\
z A_{i,j}(z)-\delta_{i,j} & =\sum_{r=1}^{q} a_{i}^{(r)} A_{i+r,j}(z)+\sum_{r=0}^{\min\{i,p\}} a_{i-r}^{(-r)} A_{i-r,j}(z),\label{newrel3}
\end{align}
where $\delta_{i,j}$ is the Kronecker delta.
\end{proposition}
\begin{proof}
Fix $i,j\geq 0$. It is evident that for $n=0$ we have $A_{[n,i,j]}=A_{[0,i,j]}=\delta_{i,j}$ (recall that by definition, paths of length zero have weight $1$, and $A_{[n,i,j]}=0$ if $\mathcal{D}_{[n,i,j]}=\emptyset$). Assume now that $n\geq 1$ and consider an arbitrary path $\gamma$ in the collection $\mathcal{D}_{[n,i,j]}$. There are two possible cases for the last step of $\gamma$. 
In the first case, the last step of $\gamma$ is a downstep  of the form $(n-1,j+r)\rightarrow(n,j)$ with weight $a_{j}^{(-r)}$, where $1\leq r\leq p$.
In the second case, the last step of $\gamma$ is an upstep or level step of the form $(n-1,j-r)\rightarrow(n,j)$
 with weight $a_{j-r}^{(r)}$,   where $0\leq r\leq \min\{j,q\}$.
The part of $\gamma$ on the interval $[0,n-1]$  is a path in the collection $\mathcal{D}_{[n-1,i,j+r]}$ or in the collection 
$\mathcal{D}_{[n-1,i,j-r]}$, respectively.

From this decomposition of $\gamma\in\mathcal{D}_{[n,i,j]}$ we deduce that
\[
A_{[n,i,j]}=\sum_{\gamma\in\mathcal{D}_{[n,i,j]}}w(\gamma)=\sum_{r=1}^{p}a_{j}^{(-r)}A_{[n-1,i,j+r]}+
\sum_{r=0}^{\min\{j,q\}}a_{j-r}^{(r)} A_{[n-1,i,j-r]}.
\]
This implies
\begin{align*}
zA_{i,j}(z)-\delta_{i,j} & =\sum_{n=1}^{\infty}\frac{A_{[n,i,j]}}{z^{n}}=\sum_{n=1}^{\infty}\frac{1}{z^{n}}
\Big(\sum_{r=1}^{p}a_{j}^{(-r)}A_{[n-1,i,j+r]}+
\sum_{r=0}^{\min\{j,q\}}a_{j-r}^{(r)} A_{[n-1,i,j-r]}\Big)\\
& = \sum_{r=1}^{p} a_{j}^{(-r)}A_{i,j+r}(z)+\sum_{r=0}^{\min\{j,q\}}a_{j-r}^{(r)}A_{i,j-r}(z),
\end{align*}
so \eqref{newrel2} is justified.

Now we prove \eqref{newrel3}. Again suppose that $n\geq 1$ and consider an arbitrary path $\gamma\in\mathcal{D}_{[n,i,j]}$. The idea is to look now at the first step of $\gamma$ instead of the last one. There are two possible cases for the first step of $\gamma$. 
In the first case, the first step of $\gamma$ is an upstep  of the form $(0,i)\rightarrow(1,i+r)$ with weight $a_{i}^{(r)}$, 
where $1\leq r\leq q$.
In the second case, the first step of $\gamma$ is a downstep or level step of the form $(0,i)\rightarrow(1,i-r)$
 with weight $a_{i-r}^{(-r)}$,   where $0\leq r\leq \min\{i,p\}$.
The part of $\gamma$ on the interval $[1,n]$ is 
a horizontal translation of a path in the collection $\mathcal{D}_{[n-1,i+r,j]}$ or in the collection $\mathcal{D}_{[n-1,i-r,j]}$, 
respectively.

From this analysis we deduce that
\[
A_{[n,i,j]} =\sum_{r=1}^{q}a_{i}^{(r)}A_{[n-1,i+r,j]}+\sum_{r=0}^{\min\{i,p\}}a_{i-r}^{(-r)}A_{[n-1,i-r,j]}.
\]
and as before this easily implies \eqref{newrel3}.
\end{proof}

\section{Lattice paths and resolvents of difference operators}\label{sec:lattpathoper}

In this section we show that the formal Laurent series defined in \eqref{7-1-22-6}, \eqref{7-1-22-7}, and \eqref{7-1-22-8} can be identified as resolvent functions of certain difference operators. Again  for each $-p\leq k\leq q$, let $(a^{(k)}_{n})_{n\in\mathbb{Z}}$ be a  bi-infinite sequence of complex numbers.
Let $\{e_{n}\}_{n \in \mathbb{Z}}$ denote the standard basis vectors in the space $\ell^{2}(\mathbb{Z})$ with the inner product 
$\langle \cdot , \cdot\rangle$. Let $\mathcal{E}$ denote the subspace consisting of all finite linear combinations of the basis vectors
$e_{n}$.

Let $\mathcal{W}$ be the linear operator (possibly unbounded) on $\mathcal{E}$ that acts on the standard basis vectors as follows
\begin{equation}\label{8-26-22-2}
\mathcal{W} e_{n}=\sum_{m=1}^{q} a_{n-m}^{(m)}\,e_{n-m}+\sum_{m=0}^{p} a_{n}^{(-m)}\,e_{n+m},\quad n\in\mathbb{Z},
\end{equation}
and is extended by linearity to $\mathcal{E}$.
In the basis $\{e_{n}\}_{n \in \mathbb{Z}}$, the matrix representation of this two-sided difference operator is the bi-infinite banded matrix $W$ in \eqref{matrixW}. We introduce a system of resolvent functions $\psi_{i,j}(z)$ associated with the operator $\mathcal{W}$ by the formulas
\begin{equation}\label{8-26-22-3}
\psi_{i,j}(z):=\sum_{n=0}^{\infty}\frac{\langle\mathcal{W}^{n} e_{j},
e_{i}\rangle}{z^{n+1}}, \qquad i,j  \in \mathbb{Z},
\end{equation}
where the series are understood as formal Laurent series in the variable $z$ with coefficients $\langle \mathcal{W}^{n} e_{j},e_{i}\rangle$. 

Let $\mathcal{E}_{0}$ denote the subspace of $\mathcal{E}$ consisting of all finite linear combinations of the basis vectors  $\{e_{n}\}_{n=0}^{\infty}$. Let $\mathcal{H}:\mathcal{E}_{0}\rightarrow\mathcal{E}_{0}$ be the linear operator (possibly unbounded) defined by 
\[
\begin{cases}
\mathcal{H} e_{0}=\sum_{m=0}^{p} a_{0}^{(-m)} e_{m},\\[0.3em]
\mathcal{H} e_{n}=\sum_{m=1}^{n} a_{n-m}^{(m)}\,e_{n-m}+\sum_{m=0}^{p} a_{n}^{(-m)}\,e_{n+m},\quad 0< n< q,\\[0.3em]
\mathcal{H} e_{n}=\sum_{m=1}^{q} a_{n-m}^{(m)}\,e_{n-m}+\sum_{m=0}^{p} a_{n}^{(-m)}\,e_{n+m},\quad n\geq q.
\end{cases}
\]
In the basis $\{e_{n}\}_{n=0}^{\infty}$, the matrix representation of this one-sided difference operator $\mathcal{H}$ is the infinite banded matrix $H$ in \eqref{matrixH}. We introduce the resolvent functions
\begin{equation}\label{7-9-22-2}
\phi_{i,j}(z):=\sum_{n=0}^{\infty}\frac{\langle \mathcal{H}^{n} e_{j},e_{i}\rangle}{z^{n+1}},\qquad
 i,j \ge 0,
\end{equation}
where the series is again understood as a formal Laurent series.

Finally, let $\mathcal{H}_{1}:\mathcal{E}_{0}\rightarrow\mathcal{E}_{0}$ be the linear operator defined by 
\[
\begin{cases}
\mathcal{H}_{1} e_{0}=\sum_{m=0}^{p} a_{1}^{(-m)} e_{m},\\[0.3em]
\mathcal{H}_{1} e_{n}=\sum_{m=1}^{n} a_{n-m+1}^{(m)}\,e_{n-m}+\sum_{m=0}^{p} a_{n+1}^{(-m)}\,e_{n+m},\quad 0< n< q,\\[0.3em]
\mathcal{H}_{1} e_{n}=\sum_{m=1}^{q} a_{n-m+1}^{(m)}\,e_{n-m}+\sum_{m=0}^{p} a_{n+1}^{(-m)}\,e_{n+m},\quad n\geq q.
\end{cases}
\]
In the basis $\{e_{n}\}_{n=0}^{\infty}$, the matrix representation of the operator $\mathcal{H}_{1}$  is the banded matrix
\[
H^{[1]}=\begin{pmatrix}
a_{1}^{(0)} & \ldots & a_{1}^{(q)} & & & \\
\vdots & a_{2}^{(0)} & \ldots & a_{2}^{(q)} & & \\
a_{1}^{(-p)} & \vdots & a_{3}^{(0)} & \ldots & a_{3}^{(q)} & \\
 & a_{2}^{(-p)} & \vdots & \ddots & & \ddots \\
 &  & a_{3}^{(-p)} & &  \ddots \\
 & & & \ddots &    
\end{pmatrix}.
\]
Note that $H^{[1]}$ is the infinite matrix obtained by removing the first row and the first column of the matrix $H$. The resolvent  functions $\phi_{i,j}^{(1)}(z)$ associated with the operator  $\mathcal{H}_{1}$ are given by
\begin{equation}
\phi_{i,j}^{(1)}(z):=\sum_{n=0}^{\infty}\frac{\langle\mathcal{H}_{1}^{n}\, e_{j},e_{i}\rangle}{z^{n+1}},\qquad i,j \geq 0.\label{8-26-22-4}\\
\end{equation}

Now we prove that the resolvent functions $\phi_{i,j}(z)$, $\phi_{i,j}^{(1)}(z)$, $\psi_{i,j}(z)$ coincide with the Laurent series $A_{i,j}(z)$, $A_{i,j}^{(1)}(z)$, $W_{i,j}(z)$ constructed in terms of the weight polynomials associated with the lattice paths defined in Section \ref{Lattice paths}.

\begin{theorem}\label{theorem2}
We have the following identities:
\begin{align}
\phi_{i,j}(z) & =A_{i,j}(z), \qquad i,j \ge 0,  \label{6-26-22-1}\\
\phi_{i,j}^{(1)}(z) & =A_{i,j}^{(1)}(z), \qquad i,j \ge 0, \label{6-26-22-2}\\
\psi_{i,j}(z) & = W_{i,j}(z) \qquad i,j \in \mathbb{Z}. \label{6-26-22-3}  
\end{align} 
\end{theorem}
\begin{proof}
First we prove \eqref{6-26-22-1}. In view of \eqref{7-9-22-2} and \eqref{7-1-22-6}, we need to show that for every $n\geq 0$ we have 
\begin{equation}\label{7-9-22-1}
\langle{\mathcal{H}}^{n}e_{j},e_{i}\rangle=A_{[n,i,j]},
\end{equation}
where $A_{[n,i,j]}$ is the weight polynomial associated with the collection 
$\mathcal{D}_{[n, i, j]}$ of all lattice paths of length $n$, with initial
point $(0,i)$ and final point $(n,j)$ with no point below the $x$-axis. 

Recall that $H$ is the matrix representation of the operator $\mathcal{H}$. We can express the entries \eqref{def:entriesH} of the matrix $H$ as follows
\[
\begin{cases}
h_{i,j}=a_{\min (i,j)}^{(j-i)}, & -p \leq j-i\leq q, \ \ i, j \geq 0,\\
h_{i,j}=0, & \mbox{otherwise}.
\end{cases}
\]
The relations $A_{[0,i,j]}=\delta_{i,j}=\langle e_{j},e_{i}\rangle$ show that \eqref{7-9-22-1} 
is valid for $n=0$. Fix integers $n\geq 1$ and $i, j\geq 0$. Set  $i_0=i$ and $i_n= j$. Writing out the matrix multiplication explicitly, we have
\begin{equation} \nonumber
\langle{\mathcal{H}}^{n}\, e_{j},e_{i}\rangle=(H^n)_{i,j}=\sum_{i_{1}, \ldots, i_{n-1}}h_{i_{0},i_{1}} h_{i_{1},i_{2}}\cdots 
h_{i_{n-2},i_{n-1}} h_{i_{n-1},i_{n}}.
\end{equation}
So, we can write
\begin{equation} \label{7-9-22-10}
\langle{\mathcal{H}}^{n} e_{j},e_{i}\rangle=\sum_{i_{1}, \ldots, i_{n-1}}a_{\min (i_{0},i_{1})}^{(i_{1}-i_{0})}a_{\min (i_{1},i_{2})}^{(i_{2}-i_{1})} \cdots
 a_{\min (i_{n-2},i_{n-1})}^{(i_{n-1}-i_{n-2})}a_{\min (i_{n-1},i_{n})}^{(i_{n}-i_{n-1})},
\end{equation}
where  
\begin{equation}\label{condlattice}
-p\leq i_{k+1}-i_{k}\leq q, \qquad \min(i_{k},i_{k+1})\geq 0,\qquad \mbox{for all}\,\,\,0 \leq k \leq n-1.
\end{equation}
Note that 
$$a_{\min (i_{k},i_{k+1})}^{(i_{k+1}-i_{k})}$$
is the weight of the step that starts at  the point $(k,i_k)$ and ends at the point $(k+1,i_{k+1})$, see \eqref{6-29-22-2}.
Indeed, if $i_{k+1} \geq i_{k}$, then 
$$a_{\min (i_{k},i_{k+1})}^{(i_{k+1}-i_{k})}=a_{i_{k}}^{(i_{k+1}-i_{k})}$$
is the weight of the upstep or level step with initial
point $(k,i_k)$ and final point $(k+1,i_{k+1}).$ In the case $i_{k}> i_{k+1}$, then
$$a_{\min (i_{k},i_{k+1})}^{(i_{k+1}-i_{k})}=a_{i_{k+1}}^{(i_{k+1}-i_{k})}$$
is the weight of the downstep with initial
point $(k,i_k)$ and final point $(k+1,i_{k+1}).$
Thus, the product 
$$a_{\min (i_{0},i_{1})}^{(i_{1}-i_{0})}a_{\min (i_{1},i_{2})}^{(i_{2}-i_{1})} \cdots
 a_{\min (i_{n-2},i_{n-1})}^{(i_{n-1}-i_{n-2})}a_{\min (i_{n-1},i_{n})}^{(i_{n}-i_{n-1})},$$
where \eqref{condlattice} holds, is the weight of a lattice path of length $n$ with initial point $(0,i)$, final point $(n,j)$, with no point below the $x$-axis, that is, a path in $\mathcal{D}_{[n,i,j]}$. Now, considering that by \eqref{7-9-22-10} the expression $\langle{\mathcal{H}}^{n}e_{j},e_{i}\rangle$
equals the sum of such products, and there is a one-to-one correspondence between paths in $\mathcal{D}_{[n,i,j]}$ and choices of $i_{1},\ldots,i_{n-1}$ satisfying \eqref{condlattice}, we get \eqref{7-9-22-1}
followed by \eqref{6-26-22-1}.

We can prove \eqref{6-26-22-3} in a similar fashion. The function $\psi_{i,j}(z)$ is given by the formal Laurent series \eqref{8-26-22-3}. The entries \eqref{def:entriesW} of the matrix $W=(w_{i,j})_{i,j \in \mathbb{Z}}$ that represents the operator $\mathcal{W}$ can be expressed as
\[
\begin{cases}
w_{i,j}=a_{\min (i,j)}^{(j-i)}, & -p \leq j-i\leq q,\\
w_{i,j}=0, & \mbox{otherwise}.
\end{cases}
\]
Using the argument above we obtain that
for any nonnegative integer $n$ and integers $i$ and $j$ we have
\begin{equation}\label{7-9-22-4}
\langle{\mathcal{W}}^{n}e_{j},e_{i}\rangle=(W^n)_{i,j}=W_{[n,i,j]},
\end{equation}
where $W_{[n,i,j]}$ is the weight polynomial associated with
the collection 
$\mathcal{P}_{[n, i, j]}$ of all lattice paths of length $n$, with initial
point $(0,i)$ and final point $(n,j)$. 
Then formula \eqref{6-26-22-3} follows directly from  the definition
\eqref{7-1-22-8} of the Laurent series $W_{i,j}(z)$, \eqref{8-26-22-3}, and \eqref{7-9-22-4}.

To prove \eqref{6-26-22-2} we note the following. If in the formula \eqref{7-1-22-6} for the Laurent series $A_{i,j}(z)$ we replace the matrix $H$ by the matrix $H^{[1]}$, we obtain the Laurent series $A_{i,j}^{(1)}(z)$. Furthermore, if
in the formulas for the resolvent functions $\phi_{i,j}(z)$ we use the operator $\mathcal{H}_{1}$ instead of the operator $\mathcal{H}$ and we replace the matrix $H$ by $H^{[1]}$, we get the resolvent functions $\phi_{i,j}^{(1)}(z)$. So \eqref{6-26-22-2} follows directly from \eqref{6-26-22-1}.
\end{proof}

\section{Resolvents of the operators $\mathcal{H}$ and $\mathcal{H}_{1}$}\label{sec:rfH}

Recall that $\mathcal{E}_{0}$ denotes the vector space of all finite linear combinations of the standard basis vectors $\{e_{n}\}_{n\geq 0}\subset\ell^{2}(\mathbb{Z})$. If $L:\mathcal{E}_{0}\rightarrow\mathcal{E}_{0}$ is a linear operator, the matrix representation of $L$ in the basis $\{e_{n}\}_{n\geq 0}$ is the matrix $(a_{i,j})_{i,j\geq 0}=(\langle L e_{j},e_{i}\rangle)_{i,j\geq 0}$. It is clear that a matrix $(a_{i,j})_{i,j\geq 0}$ is the matrix representation of an operator on $\mathcal{E}_{0}$ if and only if every column of the matrix has finitely many non-zero entries. 

Let $\mathcal{L}_{0}((z^{-1}))$ denote the set of all formal Laurent series
\[
A(z)=\sum_{n\in\mathbb{Z}}A_{n} z^{n}
\]
whose coefficients $A_{n}$ are linear operators on $\mathcal{E}_{0}$ and have only finitely many non-zero coefficients with index $n>0$. This set is a non-commutative ring with the usual addition and multiplication
\begin{align*}
(A+B)(z) & =\sum_{n\in\mathbb{Z}}(A_{n}+B_{n})\,z^{n},\\
(A\cdot B)(z) & =\sum_{n\in\mathbb{Z}}\Big(\sum_{l\in\mathbb{Z}} A_{l}\,B_{n-l}\Big) z^{n}.
\end{align*}
The reader can easily check that if $L$ is any operator on $\mathcal{E}_{0}$, then $zI-L$ is an invertible element in the ring and 
\begin{equation}\label{seriesinvform}
(zI-L)^{-1}=\sum_{n=0}^{\infty}\frac{L^{n}}{z^{n+1}}.
\end{equation}
More generally we have the following.

\begin{lemma}\label{lemmainversion}
A series of the form
\[
B(z)=zI+\sum_{n=0}^{\infty}\frac{B_{n}}{z^{n}}\in\mathcal{L}_{0}((z^{-1}))
\]
is invertible.
\end{lemma}
\begin{proof}
We need to show that there exists $A(z)=I z^{-1}+\sum_{n=2}^{\infty}A_{n}z^{-n}\in\mathcal{L}_{0}((z^{-1}))$ such that $B(z)A(z)=A(z)B(z)=I$. For a series $A(z)$ as indicated, we have
\[
B(z)A(z)=I+(B_{0}+A_{2})z^{-1}+\sum_{n=2}^{\infty}(B_{n-1}+A_{n+1}+\sum_{j=0}^{n-2}B_{j}A_{n-j})z^{-n}.
\]
So if we define recursively the coefficients $A_{n}$, $n\geq 2$, by the formulas
\begin{align*}
A_{2} & :=-B_{0}\\
A_{n+1}& :=-B_{n-1}-\sum_{j=0}^{n-2}B_{j} A_{n-j},\qquad n\geq 2,
\end{align*}
then $B(z) A(z)=I$.
Analogously, there exists $C(z)=I z^{-1}+\sum_{n=2}^{\infty}C_{n}z^{-n}\in\mathcal{L}_{0}((z^{-1}))$ such that $C(z)B(z)=I$.
As above, if we define recursively the coefficients $C_{n}$, $n\geq 2$, by the formulas
\begin{align*}
C_{2} & :=-B_{0}\\
C_{n+1}& :=-B_{n-1}-\sum_{j=0}^{n-2}C_{n-j} B_{j},\qquad n\geq 2,
\end{align*}
then $C(z)B(z)=I$.
So we now have  $B(z)A(z)=I$ and $C(z)B(z)=I$. Then 
$$C(z)=C(z)I=C(z)(B(z)A(z))=(C(z)B(z))A(z)=A(z),$$
which concludes the proof.
\end{proof}

Recall that $\mathbb{C}((z^{-1}))$ denotes the set of all scalar formal Laurent series with complex coefficients. We say that a matrix $(f_{i,j}(z))_{i,j\geq 0}$ with entries in $\mathbb{C}((z^{-1}))$ is the matrix representation of the series $A(z)=\sum_{n\in\mathbb{Z}} A_{n} z^n \in \mathcal{L}_{0}((z^{-1}))$ if $f_{i,j}(z)=\sum_{n\in\mathbb{Z}}\langle A_{n} e_{j},e_{i}\rangle z^n$ for all $i,j\geq 0$. We indicate that relation by writing
\[
\mathcal{M}_{A}(z)=(f_{i,j}(z))_{i,j\geq 0}.
\]
It is easy to check that a matrix $(f_{i,j}(z))_{i,j\geq 0}$ of scalar formal Laurent series is the matrix representation of a series in $\mathcal{L}_{0}((z^{-1}))$ if and only if the following two conditions hold: 
\begin{itemize}
\item[1)] There exists $d\in\mathbb{Z}$ such that $\deg(f_{i,j}(z))\leq d$ for all $i,j\geq 0$.
\item[2)] For every $n\leq d$ and $j\geq 0$, there exists $\ell\geq 0$ such that $[z^{n}] f_{i,j}(z)=0$ for all $i\geq \ell$.
\end{itemize} 
Observe that if $A(z),B(z)\in\mathcal{L}_{0}((z^{-1}))$ and $C(z)=A(z)B(z)$, then $\mathcal{M}_{C}(z)=\mathcal{M}_{A}(z) \mathcal{M}_{B}(z)$, so matrix representations of formal Laurent series can be multiplied in the usual manner.

The map $A(z)\mapsto \mathcal{M}_{A}(z)$ is trivially injective. As a result, if $A(z), B(z)\in\mathcal{L}_{0}((z^{-1}))$ satisfy $\mathcal{M}_{A}(z) \mathcal{M}_{B}(z)=\mathcal{M}_{B}(z) \mathcal{M}_{A}(z)=I=\mathcal{M}_{I}(z)$, then $A(z)B(z)=B(z)A(z)=I$, so $B(z)=A^{-1}(z)$.

It follows from \eqref{7-9-22-2}, \eqref{8-26-22-4}, and \eqref{seriesinvform} that the matrix representation of $(zI-\mathcal{H})^{-1}$ is the matrix $(\phi_{i,j}(z))_{i,j\geq 0}$, and the matrix representation of $(zI-\mathcal{H}_{1})^{-1}$ is $(\phi_{i,j}^{(1)}(z))_{i,j\geq 0}$. Note that the following theorem is a direct consequence of Theorems \ref{theorem1} and \ref{theorem2}. However, we want to give an alternative proof based on the preceding ideas. We also mention that an alternative proof of Proposition \ref{pr5-18-23} follows directly from Theorem \ref{theorem2} by applying \eqref{6-26-22-1}, \eqref{7-9-22-2}, and the fact that $(zI-\mathcal{H})(zI-\mathcal{H})^{-1}=I$.

\begin{theorem}\label{theorem3}
The following relations hold between the resolvent functions defined  in \eqref{7-9-22-2} and \eqref{8-26-22-4}: 
\begin{align}
\phi_{0,0}(z) & =\frac{1}{z-a_{0}^{(0)}-\sum_{i=1}^{q}\sum_{j=1}^{p}a_{0}^{(i)}a_{0}^{(-j)}\,\phi_{i-1,j-1}^{(1)}(z)},\label{6-4-1}\\
\phi_{0,j}(z) & =\phi_{0,0}(z)\sum_{i=1}^{q} a_{0}^{(i)}\phi_{i-1,j-1}^{(1)}(z), \qquad j\geq 1,\label{6-4-2}\\
\phi_{i,0}(z) & =\phi_{0,0}(z)\sum_{j=1}^{p} a_{0}^{(-j)}\phi_{i-1,j-1}^{(1)}(z),\qquad i\geq 1,\label{6-4-3}\\
\phi_{i,j}(z) & =\frac{\phi_{i,0}(z)\,\phi_{0,j}(z)}{\phi_{0,0}(z)}+\phi^{(1)}_{i-1,j-1}(z),\qquad i,j\geq 1.\label{6-4-4}\
\end{align}
\end{theorem}
\begin{proof}
We consider first the matrix 
\[
zI-H=\begin{pmatrix}
z-a_{0}^{(0)} & \cdots & -a_{0}^{(q)} & & & \\
\vdots & z-a_{1}^{(0)} & \cdots & -a_{1}^{(q)} & &\\
-a_{0}^{(-p)} & \vdots & z-a_{2}^{(0)} & \cdots & -a_{2}^{(q)} &\\
 & -a_{1}^{(-p)} & \vdots & \ddots & & \ddots \\
 &  & -a_{2}^{(-p)} & & \ddots \\
 & & & \ddots &  
\end{pmatrix}
\]
which represents $zI-\mathcal{H}$. We partition this matrix into blocks
\[
zI-H=\begin{pmatrix}
A & B \\
C & D 
\end{pmatrix},
\]
where $A$ is the $1\times1$ matrix
\begin{equation}\label{8-31-22-2A}
A=\begin{pmatrix}
z-a_0^{(0)}
\end{pmatrix}.
\end{equation}
So the block $B$ is the infinite row vector   
\[
B=\begin{pmatrix}
-a_{0}^{(1)} & \cdots & -a_{0}^{(q)} & 0 & 0 & 0& \cdots\\
\end{pmatrix}
\]
with entries
\begin{equation}\label{8-31-22-2B}
\begin{cases}
b_{j}=-a_{0}^{(j)}, & 1\leq j\leq q,\\
b_{j}=0, & \mbox{otherwise}.
\end{cases}
\end{equation}
$C$ is the infinite column vector
\[
C=\begin{pmatrix}
-a_{0}^{(-1)}\,\\
\vdots\,\\
-a_{0}^{(-p)}\,\\
 0\,\\
0\,\\
0\,\\
\vdots\\
\end{pmatrix}
\]
with entries
\begin{equation}\label{8-31-22-2C}
\begin{cases}
c_{i}=-a_{0}^{(-i)}, & 1\leq i\leq p,\\
c_{i}=0, & \mbox{otherwise}.
\end{cases}
\end{equation}
The block $D=zI-{H}^{[1]}$ is the matrix
\[
\begin{pmatrix}
z-a_{1}^{(0)} & \cdots & -a_{1}^{(q)} & & & \\
\vdots & z-a_{2}^{(0)} & \cdots & -a_{2}^{(q)} & &\\
-a_{1}^{(-p)} & \vdots & z-a_{3}^{(0)} & \cdots & -a_{3}^{(q)} &\\
 & -a_{2}^{(-p)} & \vdots & \ddots & & \ddots \\
 &  & -a_{3}^{(-p)} & & \ddots \\
 & & & \ddots &  
\end{pmatrix}.
\]
The series $zI-\mathcal{H}_{1}$ is invertible, so $D$ is invertible and 
\begin{equation}\label{8-31-22-55}
\mathcal{M}_{(zI-\mathcal{H}_{1})^{-1}}(z)=(\phi_{i-1,j-1}^{(1)}(z))_{i,j=1}^{\infty}=D^{-1}.
\end{equation}
From \eqref{8-31-22-2C} and \eqref{8-31-22-55} we obtain that $D^{-1}C=((D^{-1}C)_{i})_{i=1}^{\infty}$ is the infinite column vector
\[
\begin{pmatrix}
-\sum_{j=1}^{p} a_{0}^{(-j)}\phi_{0,j-1}^{(1)}(z)\\
\vdots \\ -\sum_{j=1}^{p} a_{0}^{(-j)}\phi_{i-1,j-1}^{(1)}(z)  \\
\vdots\\
\\
\end{pmatrix}
\]
with entries
\begin{equation}\label{8-31-22-2D^-1C} 
(D^{-1}C)_{i}= -\sum_{j=1}^{p} a_{0}^{(-j)}\phi_{i-1,j-1}^{(1)}(z), \qquad i \geq 1.
\end{equation}
Note that $A-BD^{-1}C$ is of size $1\times 1$. Using \eqref{8-31-22-2A}, \eqref{8-31-22-2B}, and \eqref{8-31-22-2D^-1C} we obtain
$$
A-BD^{-1}C=z-a_{0}^{(0)}-\sum_{i=1}^{q}\sum_{j=1}^{p}a_{0}^ {(i)}a_{0}^{(-j)}\,\phi_{i-1,j-1}^{(1)}(z).
$$
This series is clearly invertible, so
\begin{equation}\label{invSchur}
(A-BD^{-1}C)^{-1}=\frac{1}{z-a_{0}^{(0)}-\sum_{i=1}^{q}\sum_{j=1}^{p}a_{0}^ {(i)}a_{0}^{(-j)}\,\phi_{i-1,j-1}^{(1)}(z)}.
\end{equation}
The matrix $BD^{-1}=((BD^{-1})_{j})_{j=1}^{\infty}$ is the infinite row vector
\[
\begin{pmatrix}
-\sum_{i=1}^{q} a_{0}^{(i)}\phi_{i-1,0}^{(1)}(z) & \cdots & -\sum_{i=1}^{q} a_{0}^{(i)}\phi_{i-1,j-1}^{(1)}(z)  &\cdots & 
\end{pmatrix}
\]
with entries
\begin{equation}\label{8-31-22-2BD^-1} 
(BD^{-1})_{j}= -\sum_{i=1}^{q} a_{0}^{(i)}\phi_{i-1,j-1}^{(1)}(z), \qquad j \geq 1.
\end{equation}

Consider the expression
\begin{gather}
\begin{pmatrix}
I & 0 \\[0.1em]
-D^{-1}C & I
\end{pmatrix}
\begin{pmatrix}
(A-BD^{-1}C)^{-1}& 0\\[0.1em]
0 & D^{-1}
\end{pmatrix}
\begin{pmatrix}
I & -BD^{-1} \\[0.1em]
0 & I
\end{pmatrix}\notag\\
=\begin{pmatrix}
(A-BD^{-1}C)^{-1}& -(A-BD^{-1}C)^{-1}BD^{-1}\\[0.1em]
-D^{-1}C(A-BD^{-1}C)^{-1}& D^{-1}C(A-BD^{-1}C)^{-1}BD^{-1}+D^{-1}
\end{pmatrix}.\label{8-31-22-22}
\end{gather}
It is easy to see that each one of the three factors on the left-hand side is the matrix representation of a series in $\mathcal{L}_{0}((z^{-1}))$, so the product is the matrix representation of a series in $\mathcal{L}_{0}((z^{-1}))$. Moreover, the matrix \eqref{8-31-22-22} is the inverse of 
\[
\begin{pmatrix}
A & B \\
C & D 
\end{pmatrix},
\]
therefore \eqref{8-31-22-22} is the matrix representation of $(zI-\mathcal{H})^{-1}$ and so we obtain
\begin{equation}\label{invblock}
(\phi_{i,j}(z))_{i,j\geq 0}=\begin{pmatrix}
(A-BD^{-1}C)^{-1}& -(A-BD^{-1}C)^{-1}BD^{-1}\\[0.1em]
-D^{-1}C(A-BD^{-1}C)^{-1}& D^{-1}C(A-BD^{-1}C)^{-1}BD^{-1}+D^{-1}
\end{pmatrix}.
\end{equation}
From \eqref{invSchur} and \eqref{invblock} we get
\[
\phi_{0,0}(z) =(A-BD^{-1}C)^{-1}=\frac{1}{z-a_{0}^{(0)}-\sum_{ i=1}^{q}\sum_{j=1}^{p}a_{0}^{(i)}a_{0}^{(-j)}\,\phi_{i-1,j -1}^{(1)}(z)},
\]
which proves \eqref{6-4-1}.

By \eqref{invblock} we have
$$
\phi_{0,j}(z)=(-(A-BD^{-1}C)^{-1}BD^{-1})_j=-\phi_{0,0}(z)( BD^{-1})_j, \quad j \ge 1.
$$
Hence, applying \eqref{8-31-22-2BD^-1} we obtain
\[
\phi_{0,j}(z)=\phi_{0,0}(z)\sum_{i=1}^{q} a_{0}^{(i)}\phi_{i-1, j-1}^{(1)}(z), \quad j \ge 1,
\]
which establishes \eqref{6-4-2}.

By \eqref{invblock} we have
$$
\phi_{i,0}(z)=(-D^{-1}C(A-BD^{-1}C)^{-1})_i=-\phi_{0,0}(z) (D^{-1}C)_i, \quad i \ge 1.
$$
So using \eqref{8-31-22-2D^-1C} we get
\begin{equation}\label{9-9-22-2}
\phi_{i,0}(z)=-\phi_{0,0}(z) (D^{-1}C)_i=\phi_{0,0}(z)\sum_{j=1}^{p} a_{0}^{(-j)}\phi_{i-1 ,j-1}^{(1)}(z), \quad i \ge 1,
\end{equation}
which justifies \eqref{6-4-3}.

Finally we prove \eqref{6-4-4}. Let $i,j\geq 1$. From \eqref{invblock} we get
\[
\phi_{i,j}(z)=(D^{-1}C(A-BD^{-1}C)^{-1}BD^{-1}+D^{-1})_{i,j}.
\]
Therefore, applying formula \eqref{9-9-22-2} for the entries of $D^{-1}C(A-BD^{-1}C)^{-1}$ and formulas \eqref{8-31-22-2BD^-1} and \eqref{8-31-22-55} for the  entries of $BD^{-1}$ and $ D^{-1}$ we obtain
\[
\phi_{i,j}(z)=\phi_{0,0}(z)(\sum_{j=1}^{p} a_{0}^{(-j)}\phi_{i-1 ,j-1}^{(1)}(z))
(\sum_{i=1}^{q} a_{0}^{(i)}\phi_{i-1,j-1}^{(1)}(z))+\phi_{i-1 ,j-1}^{(1)}(z),
\]
which implies (see \eqref{6-4-2} and \eqref{6-4-3})
\[
\phi_{i,j}(z) =\frac{\phi_{i,0}(z)\,\phi_{0,j}(z)}{\phi_{0,0}(z)}+ \phi^{(1)}_{i-1,j-1}(z), \quad i,j\ge 1.
\]
So \eqref{6-4-4} is proved.
\end{proof}

\section{Resolvents of the operator $\mathcal{W}$}\label{Wmatrix} 
\subsection{Relations between resolvents of one-sided and two-sided operators}

Let $\{e_{n}\}_{n\in\mathbb{Z}}$ denote the standard basis vectors in the space $\ell^{2}(\mathbb{Z})$ with the inner product $\langle\cdot,\cdot\rangle$. Recall that $\mathcal{E}=\mbox{span}\{e_{n}\}_{n\in\mathbb{Z}}$, $\mathcal{E}_{0}=\mbox{span}\{e_{n}\}_{n\geq 0}$, and let $\mathcal{E}_{1}=\mbox{span}\{e_{n}\}_{n\leq -1}$. Consider the spaces of linear operators  $\mathcal{L}_{i,j}=\{L:\mathcal{E}_{i}\rightarrow\mathcal{E}_{j}\}$, $0\leq i,j\leq 1$, and $\mathcal{L}=\{L: \mathcal{E}\rightarrow\mathcal{E}\}$. 

Let $\mathcal{L}((z^{-1}))$ denote the set of all formal Laurent series
\[
A(z)=\sum_{n\in\mathbb{Z}} A_{n} z^{n}
\]
with coefficients $A_{n}\in\mathcal{L}$, and having only finitely many non-zero coefficients with index $n>0$. Similarly we define the sets $\mathcal{L}_{i,j}((z^{-1}))$ of series with coefficients in $\mathcal{L}_{i,j}$. If $i=j$, we write $\mathcal{L}_{i}((z^{-1}))=\mathcal{L}_{i,i}((z^{-1}))$. If $A(z)\in\mathcal{L}_{j,k}((z^{-1}))$ and $B(z)\in\mathcal{L}_{i,j}((z^{-1}))$, then 
\[
(A\cdot B)(z)=\sum_{n\in\mathbb{Z}}\Big(\sum_{l\in\mathbb{Z}} A_{l}\, B_{n-l}\Big)z^{n}\in\mathcal{L}_{i,k}((z^{-1})).
\]
An operator $L\in\mathcal{L}$ has the two-sided matrix representation $(\langle L e_{j},e_{i}\rangle)_{i,j\in\mathbb{Z}}$. Similarly, the matrix representation of an operator $L\in\mathcal{L}_{i,j}$ is a matrix with rows and columns indexed by the indices in the bases for $\mathcal{E}_{j}$ and $\mathcal{E}_{i}$ respectively. The matrix representation of a series $A(z)=\sum_{n\in\mathbb{Z}} A_{n} z^{n}\in\mathcal{L}((z^{-1}))$ is
\[
\mathcal{M}_{A}(z)=(f_{i,j}(z))_{i,j\in\mathbb{Z}}=(\sum_{n\in\mathbb{Z}}\langle A_{n} e_{j},e_{i}\rangle z^{n})_{i,j\in\mathbb{Z}}
\]
and similarly for the other sets of formal Laurent series. If the product $A\cdot B$ is defined for a pair of formal Laurent series $A(z)$ and $B(z)$, then $\mathcal{M}_{A\cdot B}(z)=\mathcal{M}_{A}(z) \mathcal{M}_{B}(z)$.

Observe also that a matrix $(f_{i,j}(z))_{i,j\in\mathbb{Z}}$ of scalar Laurent series is the matrix representation of a series in $\mathcal{L}((z^{-1}))$ if and only if the following two properties hold:
\begin{itemize}
\item[1)] There exists $d\in\mathbb{Z}$ such that $\deg(f_{i,j}(z))\leq d$ for all $i,j\in\mathbb{Z}$.
\item[2)] For every $n\leq d$ and $j\in\mathbb{Z}$, there exists $\ell\geq 0$ such that $[z^{n}]f_{i,j}(z)=0$ for all $i\in\mathbb{Z}$, $|i|\geq \ell$.
\end{itemize} 

The operator $\mathcal{W}\in\mathcal{L}$ was defined in \eqref{8-26-22-2}. Its matrix representation is the two-sided matrix $W=(\langle\mathcal{W} e_{j},e_{i}\rangle)_{i,j\in\mathbb{Z}}$ in \eqref{matrixW}. We partition the matrix $zI-W$ in block form
\begin{equation}\label{9-21-22-1}
zI-W=\begin{pmatrix}
A & B \\
C & D \\
\end{pmatrix},
\end{equation}
where $A,B,C,D$ are matrices described below. The matrix $A$ is the following infinite matrix 
\begin{equation}\label{defmatrixA}
A=\begin{pmatrix}
   &  & \ddots & &\,\,\\
  & \ddots & &- a_{-q-2}^{(q)} & \,\,\\
 \ddots&   & \ddots & \vdots & -a_{-q-1}^{(q)} \,\,\\
   & -a_{-p-2}^{(-p)} & \ldots & z-a_{-2}^{(0)}& \vdots\,\,\\[0.5em]
 & &-a_{-p-1}^{(-p)} & \ldots & z-a_{-1}^{(0)}\,\, 
\end{pmatrix}.
\end{equation}
We use indices $i,j\leq -1$ to label the entries of $A$, which we represent in the form $A=zI-W_{1}=((zI-W_{1})_{i,j})_{i,j \le -1}$. So $W_{1}=((W_{1})_{i,j})_{i,j \le -1}$ is the block of $W$
\begin{equation}\label{def:matrixW1}
W_{1}=\begin{pmatrix}
   &  & \ddots & & \,\,\\
  & \ddots & & a_{-q-2}^{(q)} & \,\,\\
\ddots &   & \ddots & \vdots & a_{-q-1}^{(q)} \,\,\\
   & a_{-p-2}^{(-p)} & \ldots & a_{-2}^{(0)}& \vdots \,\,\\[0.5em]
 & &a_{-p-1}^{(-p)} & \ldots & a_{-1}^{(0)} \,\,  
\end{pmatrix}
\end{equation}
with entries 
\[
\begin{cases}
(W_{1})_{i,j}=a_{\min (i,j)}^{(j-i)}, & -p \leq j-i\leq q, \,\,\, i, j \leq -1,\\
(W_{1})_{i,j}=0, & \mbox{otherwise}.
\end{cases}
\]
We can see $W_{1}$ as the matrix representation of a linear operator $\mathcal{W}_{1}:\mathcal{E}_{1}\rightarrow\mathcal{E}_{1}$ in the basis $\{e_{n}\}_{n\leq -1}$. 

The matrix $B=(B_{i,j})_{ i\le -1,\,j\geq 0}$ is an infinite matrix with a triangular set of non-zero entries in the lower-left corner:
\[ 
B=\begin{pmatrix}
  &  & & & & &\\
   &  & & & & &\\
   -a_{-q}^{(q)} &  & & & & &\\
  \vdots& \ddots &  &  &  & &\\
  -a_{-2}^{(2)}  & \ddots & \ddots &  & & & &\\[0.5em]
 -a_{-1}^{(1)}  & -a_{-1}^{(2)} &\cdots &-a_{-1}^{(q)} &  & & 
\end{pmatrix}.
\]
$B$ is a block of $-W$ with entries
\begin{equation}\label{8-31-22-2BB}
\begin{cases}
B_{i,j}=-a_{i}^{(j-i)}, & 1 \leq j-i\leq q,    \ \ i\le -1, \  j \ge 0,\\
B_{i,j}=0, & \mbox{otherwise}.
\end{cases} 
\end{equation}
The matrix $B$ is clearly the matrix representation of an operator from $\mathcal{E}_{0}$ to $\mathcal{E}_{1}$.

The matrix $C=(C_{i,j})_{i\geq 0,\,j\leq -1}$ is a  matrix with a triangular set of non-zero entries in the upper-right corner:
\[
\begin{pmatrix}
& & & -a_{-p}^{(-p)}& \ldots& -a_{-2}^{(-2)}& -a_{-1}^{(-1)}\\
& &  &  & \ddots& \ddots& -a_{-1}^{(-2)}\\
& &  &  & & \ddots & \vdots \\
& & &  &  &  & -a_{-1}^{(-p)} \\
& &   &  & &  & & \\
& &  &  & & & 
\end{pmatrix}.
\]
$C$ is also a block of $-W$ with entries
\begin{equation}\label{8-31-22-2CC}
\begin{cases}
C_{i,j}=-a_{j}^{(j-i)}, & -p \leq j-i\leq -1,    \ \ i\ge 0, \  j \le -1,\\
C_{i,j}=0, & \mbox{otherwise}.
\end{cases}
\end{equation}
The matrix $C$ is the matrix representation of an operator from $\mathcal{E}_{1}$ to $\mathcal{E}_{0}$.

Finally, $D=zI-H=(zI-H)_{ i,j \ge  0}$ is the invertible matrix
\[
\begin{pmatrix}
z-a_{0}^{(0)} & \ldots & -a_{0}^{(q)} & &\\
\vdots & z-a_{1}^{(0)} & \ldots & -a_{1}^{(q)} & \\
-a_{0}^{(-p)} & \vdots & \ddots &  & \ddots\\
 & -a_{1}^{(-p)} &  & \ddots &\\
 &  & \ddots& 
\end{pmatrix}.
\]
Denote by $\chi_{i,j}(z)$ the system of resolvent functions associated with the operator $\mathcal{W}_{1}$ with matrix representation \eqref{def:matrixW1}, i.e.,
\[
\chi_{i,j}(z):=\sum_{n=0}^{\infty}\frac{\langle\mathcal{W}_{1}^{n}\, e_{j}, e_{i}\rangle}{z^{n+1}},
\qquad i,j \le -1.
\]
So by definition the inverse of the matrix $A=zI-W_{1}$ in \eqref{defmatrixA} is 
\[
A^{-1}=(zI-W_{1})^{-1}=(\chi_{i,j}(z))_{i,j\leq -1}.
\]

Let $K=(K_{i,j})_{ i,j \ge 0}$  be the banded matrix with the following entries:
\begin{equation}\label{8-31-22-101}
\begin{aligned}
K_{i,j} &=h_{i,j}+\sum ^{p-i}_{l=1}\sum^{q-j}_{m=1}a^{(-(l+i))}_{-l}a^{(m+j)}_{-m}\chi_{-l,-m}(z), \qquad  0\le i \le p-1,\quad  0 \le j \le
 q-1,\\
K_{i,j} &=h_{i,j}, \qquad \mbox{otherwise},
\end{aligned}
\end{equation}
where $h_{i,j}$ is the $(i,j)$-entry of $H$. Clearly, the matrix $K=(K_{i,j})_{i,j\geq 0}$ is the matrix representation of a formal Laurent series $\mathcal{K}(z)\in\mathcal{L}_{0}((z^{-1}))$ (see properties $1)$--$2)$ in Section~\ref{sec:rfH}). Moreover, the series $zI-\mathcal{K}(z)$ is invertible in $\mathcal{L}_{0}((z^{-1}))$ by Lemma~\ref{lemmainversion}. Therefore the matrix $zI-K$ is invertible, and we define the scalar series $\zeta_{i,j}(z)$ as the $(i,j)$-entry of the matrix $(zI-K)^{-1}$, i.e., we have
\begin{equation}\label{10-05-22-11}
(zI-K)^{-1}=(\zeta_{i,j}(z))_{i,j\geq 0}.
\end{equation}
By \eqref{8-26-22-3}, we also have
\[
(zI-W)^{-1}=(\psi_{i,j}(z))_{i,j\in\mathbb{Z}}.
\]
The following result states that when $i$ and $j$ are non-negative indices, the resolvent  function $\psi_{i,j}(z)$ and the $(i,j)$-entry $\zeta_{i,j}(z)$ of the matrix $(zI-K)^{-1}$ are the same.

\begin{theorem}\label{10-3-22-1}
We have the following identities:
\begin{equation}\label{10-05-22-1}
W_{i,j}(z)=\psi_{i,j}(z) =\zeta_{i,j}(z), \qquad i,j  \ge 0. \\
\end{equation}
\end{theorem}

\begin{proof}
The equality $W_{i,j}(z)=\psi_{i,j}(z)$ was already proved in Theorem \ref{theorem2}. Let's now prove the equality $\psi_{i,j}(z)=\zeta_{i,j}(z).$
The matrix $zI-W$ is invertible since it is the matrix representation of the invertible series $zI-\mathcal{W}\in\mathcal{L}((z^{-1}))$. Recall that we have the block partition \eqref{9-21-22-1}. By Lemma~\ref{lemmainversion}, the matrix $D-CA^{-1}B$ is invertible as it represents an invertible series in the space $\mathcal{L}_{0}((z^{-1}))$. 
By the same reasoning, the matrix $A-BD^{-1}C$ is also invertible. Consider the expression
\begin{gather}
\begin{pmatrix}
(A-BD^{-1}C)^{-1} & 0\\[0.1em]
0 & (D-CA^{-1}B)^{-1}
\end{pmatrix}
\begin{pmatrix}
I & -BD^{-1} \\[0.1em]
-CA^{-1} & I
\end{pmatrix}\notag\\
=\begin{pmatrix}
(A-BD^{-1}C)^{-1}& -(A-BD^{-1}C)^{-1}BD^{-1}\\[0.1em]
-(D-CA^{-1}B)^{-1}CA^{-1} & (D-CA^{-1}B)^{-1}
\end{pmatrix}.\label{blockinverse}
\end{gather}
It is easy to see that \eqref{blockinverse} is the matrix representation of a series in $\mathcal{L}((z^{-1}))$. Moreover, this matrix is the inverse of 
\[
\begin{pmatrix}
A & B \\
C & D 
\end{pmatrix},
\] 
therefore we obtain
\begin{equation}\label{10-3-22-2}
(zI-W)^{-1}=\begin{pmatrix}
A & B \\
C & D 
\end{pmatrix}^{-1}=\begin{pmatrix}
(A-BD^{-1}C)^{-1}& -(A-BD^{-1}C)^{-1}BD^{-1}\\[0.1em]
-(D-CA^{-1}B)^{-1}CA^{-1} & (D-CA^{-1}B)^{-1}
\end{pmatrix}.
\end{equation}

Recall
\begin{equation}\label{8-31-22-4}
(zI-W)^{-1}=\begin{pmatrix}
A & B \\
C & D
\end{pmatrix}^{-1}= (\psi_{i,j}(z))_{i , j \in \mathbb{Z}} 
\end{equation}
and
\begin{equation}\label{10-3-22-5}
A^{-1}=(zI-W_{1})^{-1}=
 (\chi_{i,j}(z))_{i,j\leq -1}.
\end{equation}
In virtue of \eqref{8-31-22-4}, \eqref{10-3-22-2}, and \eqref{10-05-22-11}, the proof of \eqref{10-05-22-1} reduces to show
\[
zI-K=D-CA^{-1}B,
\]
which is equivalent to
\begin{equation}\label{10-05-22-2}
K=H+CA^{-1}B.
\end{equation}
\noindent
Using \eqref{10-3-22-5} and \eqref{8-31-22-2BB}, and applying the rules of matrix multiplication, we get the following formulas
for the entries of the matrix  $A^{-1}B=((A^{-1}B)_{i,j})_{i \le -1,  j \ge 0}$:
\begin{align*}
 (A^{-1}B)_{i,j} &=-\sum^{q-j}_{m=1}a^{(m+j)}_{-m}\chi_{i,-m}(z), \quad  i \le -1,\quad  0 \le j \le q-1,\\
(A^{-1}B)_{i,j} &=0, \qquad  i \le -1, \quad j \ge q.\\
\end{align*}
Therefore, by \eqref{8-31-22-2CC}, the entries of the matrix $H+CA^{-1}B=((H+CA^{-1}B)_{i,j})_{ i,j \ge 0}$ are
\begin{align*}
(H+CA^{-1}B)_{i,j} &=h_{i,j}+\sum ^{p-i}_{l=1}\sum^{q-j}_{m=1}a^{(-(l+i))}_{-l}a^{(m+j)}_{-m}\chi_{-l,-m}(z), \quad  0\le i \le p-1,\quad  0 \le j \le q-1,\\
(H+CA^{-1}B)_{i,j} &=h_{i,j},\qquad  \mbox{otherwise.}
\end{align*}
These relations and \eqref{8-31-22-101} justify \eqref{10-05-22-2}, and this concludes the proof.
\end{proof}

\subsection{The collection of paths $\widehat{\mathcal{D}}_{[n,i,j]}$}\label{subsecDhat}

We now require an additional collection of lattice paths. For integers $n\geq 0$ and $i,j\leq -1$, the collection $\widehat{\mathcal{D}}_{[n,i,j]}$ consists of those paths in $\mathcal{P}_{[n,i,j]}$ that never go above the line $y=-1$. The weight polynomials associated with this collection are denoted by $V_{[n,i,j]}$, thus
\[
V_{[n,i,j]}:=\sum_{\gamma\in\widehat{\mathcal{D}}_{[n,i, j]}}w(\gamma).
\]  
For integers $i, j \leq -1$, we define the Laurent series $V_{i,j}(z)$ as follows:
\begin{equation}\label{def:Vijseries}
V_{i,j}(z):=\sum_{n=0}^{\infty}\frac{V_{[n,i,j]}}{z^{n+1}},
\end{equation}
where  $V_{[0,i,j]}=1$ if $i=j$, and $V_{[0,i,j]}=0$ if  $i \neq j$.
Now we assert that 
the  resolvent functions $\chi_{i,j}(z)$ associated with the operator $\mathcal{W}_{1}$ with matrix representation \eqref{def:matrixW1}
coincide with the Laurent series 
$V_{i,j}(z)$ constructed in terms of the weight polynomials associated with the collections $\widehat{\mathcal{D}}_{[n,i,j]}$ of lattice paths. 
\begin{proposition}\label{theoremV}
We have the following identities:
\begin{equation*}
\chi_{i,j}(z) =V_{i,j}(z), \qquad i,j \le -1.  
\end{equation*} 
\end{proposition}
This proposition can be proven in the same manner as we proved the first identity \eqref{6-26-22-1} in Theorem \ref{theorem2}. We will leave the details to the reader.  
\noindent
By applying this proposition and taking into account \eqref{8-31-22-101}  and the formula $h_{i,j}=a_{\min (i,j)}^{(j-i)}$ for $0\le i \le p-1$ and $0 \le j \le q-1$, we can now express the entries of the banded matrix $K$ defined in \eqref{8-31-22-101} in the following manner:
\begin{equation}\label{6-12-23-31}
\begin{aligned}
K_{i,j} &=a_{\min (i,j)}^{(j-i)}+\sum ^{p-i}_{l=1}\sum^{q-j}_{m=1}a^{(-(l+i))}_{-l}a^{(m+j)}_{-m}V_{-l,-m}(z), \qquad  0\le i \le p-1,\quad  0 \le j \le
 q-1, \\
K_{i,j} &=h_{i,j}, \qquad \mbox{otherwise},
\end{aligned}
\end{equation}
where $h_{i,j}$ is the $(i,j)$-entry of $H$.

\par

\begin{figure}
\begin{center}
\begin{tikzpicture}[scale=0.6]
\draw[line width=1.5pt]  (-3.03,0) -- (15.5,0);
\draw[line width=1.5pt]  (-3,0) -- (-3,4.5);
\draw[line width=0.7pt] (-3,0) -- (-2,1) -- (-1,3) -- (0,1) -- (1,4) -- (2,1) -- (3,0) -- (4,2) -- (5,3) -- (6,3) -- (7,4) -- (8,0) -- (9,2) -- (10,1) -- (11,0) -- (12,3) -- (13,1) -- (14,2) -- (15,3);
\draw [line width=0.5] (-2,0) -- (-2,-0.12);
\draw [line width=0.5] (-1,0) -- (-1,-0.12);
\draw [line width=0.5] (0,0) -- (0,-0.12);
\draw [line width=0.5] (1,0) -- (1,-0.12);
\draw [line width=0.5] (2,0) -- (2,-0.12);
\draw [line width=0.5] (3,0) -- (3,-0.12);
\draw [line width=0.5] (4,0) -- (4,-0.12);
\draw [line width=0.5] (5,0) -- (5,-0.12);
\draw [line width=0.5] (6,0) -- (6,-0.12);
\draw [line width=0.5] (7,0) -- (7,-0.12);
\draw [line width=0.5] (8,0) -- (8,-0.12);
\draw [line width=0.5] (9,0) -- (9,-0.12);
\draw [line width=0.5] (10,0) -- (10,-0.12);
\draw [line width=0.5] (11,0) -- (11,-0.12);
\draw [line width=0.5] (12,0) -- (12,-0.12);
\draw [line width=0.5] (13,0) -- (13,-0.12);
\draw [line width=0.5] (14,0) -- (14,-0.12);
\draw [line width=0.5] (15,0) -- (15,-0.12);
\draw [line width=0.5] (-3,0) -- (-3.12,0);
\draw [line width=0.5] (-3,1) -- (-3.12,1);
\draw [line width=0.5] (-3,2) -- (-3.12,2);
\draw [line width=0.5] (-3,3) -- (-3.12,3);
\draw [line width=0.5] (-3,4) -- (-3.12,4);
\draw [dotted] (-3,1) -- (15.5,1);
\draw [dotted] (-3,2) -- (15.5,2);
\draw [dotted] (-3,3) -- (15.5,3);
\draw [dotted] (-3,4) -- (15.5,4);
\draw [dotted] (-2,0) -- (-2,4.5);
\draw [dotted] (-1,0) -- (-1,4.5);
\draw [dotted] (0,0) -- (0,4.5);
\draw [dotted] (1,0) -- (1,4.5);
\draw [dotted] (2,0) -- (2,4.5);
\draw [dotted] (3,0) -- (3,4.5);
\draw [dotted] (4,0) -- (4,4.5);
\draw [dotted] (5,0) -- (5,4.5);
\draw [dotted] (6,0) -- (6,4.5);
\draw [dotted] (7,0) -- (7,4.5);
\draw [dotted] (8,0) -- (8,4.5);
\draw [dotted] (9,0) -- (9,4.5);
\draw [dotted] (10,0) -- (10,4.5);
\draw [dotted] (11,0) -- (11,4.5);
\draw [dotted] (12,0) -- (12,4.5);
\draw [dotted] (13,0) -- (13,4.5);
\draw [dotted] (14,0) -- (14,4.5);
\draw [dotted] (15,0) -- (15,4.5);
\draw (-2,-0.1) node[below, scale=0.8]{$1$};
\draw (-1,-0.1) node[below, scale=0.8]{$2$};
\draw (0,-0.1) node[below, scale=0.8]{$3$};
\draw (1,-0.1) node[below, scale=0.8]{$4$};
\draw (2,-0.1) node[below, scale=0.8]{$5$};
\draw (3,-0.1) node[below, scale=0.8]{$6$};
\draw (4,-0.1) node[below, scale=0.8]{$7$};
\draw (5,-0.1) node[below, scale=0.8]{$8$};
\draw (6,-0.1) node[below, scale=0.8]{$9$};
\draw (7,-0.1) node[below, scale=0.8]{$10$};
\draw (8,-0.1) node[below, scale=0.8]{$11$};
\draw (9,-0.1) node[below, scale=0.8]{$12$};
\draw (10,-0.1) node[below, scale=0.8]{$13$};
\draw (11,-0.1) node[below, scale=0.8]{$14$};
\draw (12,-0.1) node[below, scale=0.8]{$15$};
\draw (13,-0.1) node[below, scale=0.8]{$16$};
\draw (14,-0.1) node[below, scale=0.8]{$17$};
\draw (15,-0.1) node[below, scale=0.8]{$18$};
\draw (-3.1,0) node[left, scale=0.8]{$0$};
\draw (-3.1,1) node[left, scale=0.8]{$1$};
\draw (-3.1,2) node[left, scale=0.8]{$2$};
\draw (-3.1,3) node[left, scale=0.8]{$3$};
\draw (-3.1,4) node[left, scale=0.8]{$4$};
\draw[line width=1.5pt]  (-3.03,0) -- (15.5,0);
\draw[line width=1.5pt]  (-3,0) -- (-3,-6.5);
\draw[line width=0.7pt] (-3,-4) -- (-2,-3) -- (-1,-2) -- (0,-4) -- (1,-1) -- (2,-2) -- (3,-3) -- (4,-1) -- (5,-5) -- (6,-4) -- (7,-4) -- (8,-3) -- (9,-1) -- (10,-2) -- (11,-5) -- (12,-2) -- (13,-4) -- (14,-2) -- (15,-1);
\draw [line width=0.5] (-3,-1) -- (-3.12,-1);
\draw [line width=0.5] (-3,-2) -- (-3.12,-2);
\draw [line width=0.5] (-3,-3) -- (-3.12,-3);
\draw [line width=0.5] (-3,-4) -- (-3.12,-4);
\draw [line width=0.5] (-3,-5) -- (-3.12,-5);
\draw [line width=0.5] (-3,-6) -- (-3.12,-6);
\draw [dotted] (-2,0) -- (-2,-6.5);
\draw [dotted] (-1,0) -- (-1,-6.5);
\draw [dotted] (0,0) -- (0,-6.5);
\draw [dotted] (1,0) -- (1,-6.5);
\draw [dotted] (2,0) -- (2,-6.5);
\draw [dotted] (3,0) -- (3,-6.5);
\draw [dotted] (4,0) -- (4,-6.5);
\draw [dotted] (5,0) -- (5,-6.5);
\draw [dotted] (6,0) -- (6,-6.5);
\draw [dotted] (7,0) -- (7,-6.5);
\draw [dotted] (8,0) -- (8,-6.5);
\draw [dotted] (9,0) -- (9,-6.5);
\draw [dotted] (10,0) -- (10,-6.5);
\draw [dotted] (11,0) -- (11,-6.5);
\draw [dotted] (12,0) -- (12,-6.5);
\draw [dotted] (13,0) -- (13,-6.5);
\draw [dotted] (14,0) -- (14,-6.5);
\draw [dotted] (15,0) -- (15,-6.5);
\draw [dotted] (-3,-1) -- (15.5,-1);
\draw [dotted] (-3,-2) -- (15.5,-2);
\draw [dotted] (-3,-3) -- (15.5,-3);
\draw [dotted] (-3,-4) -- (15.5,-4);
\draw [dotted] (-3,-5) -- (15.5,-5);
\draw [dotted] (-3,-6) -- (15.5,-6);
\draw (-3.1,-1) node[left, scale=0.8]{$-1$};
\draw (-3.1,-2) node[left, scale=0.8]{$-2$};
\draw (-3.1,-3) node[left, scale=0.8]{$-3$};
\draw (-3.1,-4) node[left, scale=0.8]{$-4$};
\draw (-3.1,-5) node[left, scale=0.8]{$-5$};
\draw (-3.1,-6) node[left, scale=0.8]{$-6$};
\end{tikzpicture}
\end{center}
\caption{Above it is shown a path $\gamma$ in the collection $\mathcal{D}_{[18,0,3]}$, and below the reflection $\widehat{\gamma}$ in the collection $\widehat{\mathcal{D}}_{[18,-4,-1]}$.}
\label{Newplot}
\end{figure}

We remark that there is a one-to-one correspondence between paths in the collection $\mathcal{D}_{[n,i,j]}$, $i,j\geq 0$, and paths in the collection $\widehat{\mathcal{D}}_{[n,-(j+1),-(i+1)]}$. The one-to-one correspondence is established by a map $\gamma \mapsto \widehat{\gamma}$ that is defined as follows. Given a path $\gamma \in \mathcal{D}_{[n,i,j]}$, it is first reflected with respect to the real axis. The result is then reflected with respect to the vertical line $x = n/2$ and shifted $1$ unit downwards to obtain the path $\widehat{\gamma} \in \widehat{\mathcal{D}}_{[n,-(j+1), -(i+1)]}$. See an example of this transformation in Fig. \ref{Newplot}. Note that under this transformation, the image of a step that belongs to one of the sets in \eqref{defsteps} is a step that belongs to the same set. Furthermore, if a step in $\gamma$ has weight $a^{(m)}_{k}, m\ge 0,$ then its image in $\widehat{\gamma}$ will have weight $a^{(m)}_{-(k+m+1)}$. In the case where a step has a weight $a^{(-m)}_{k}, m\ge 0,$ its image will have weight $a^{(-m)}_{-(k+m+1)}$. Therefore, there exists a one-to-one correspondence 
$$A_{i,j}(z) \leftrightarrow V_{-(j+1),-(i+1)}(z), \quad \quad i,j\geq 0,$$
between the formal series $A_{i,j}(z)$ and $V_{-(j+1),-(i+1)}(z)$, and if we define the following banded matrix 
\begin{equation}\label{matrixE}
E:=\begin{pmatrix}
a_{-1}^{(0)} & \cdots & a_{-(q+1)}^{(q)} & & & \\
\vdots & a_{-2}^{(0)} & \cdots & a_{-(q+2)}^{(q)} & &\\
a_{-(p+1)}^{(-p)} & \vdots & a_{-3}^{(0)} & \cdots & a_{-(q+3)}^{(q)} &\\
 & a_{-(p+2)}^{(-p)} & \vdots & \ddots & & \ddots \\
 &  & a_{-(p+3)}^{(-p)} & & \ddots \\
 & & & \ddots &  
\end{pmatrix},
\end{equation}
then the algebraic relation between the formal series $V_{-(j+1),-(i+1)}(z)$, $i,j\geq 0$, and the matrix $E$ is the same as the relation between $A_{i,j}(z)$ and the matrix $H$. In Section \ref{MCF}, we will use the matrix $E$ to construct a matrix continued fraction expansion (see Proposition \ref{P6-10-23-1}),
for the following matrix: 
\begin{equation*}
\begin{pmatrix}
V_{-1,-1}(z) &\cdots & V_{-p, -1}(z)\\
\vdots &\ddots& \vdots\\
V_{-1, -q}(z) & \cdots& V_{-p, -q}(z)\\
\end{pmatrix}.
\end{equation*}

\section{Matrix continued fractions}\label{MCF}

\subsection{Main results}\label{subsecMCF1}

Consider a $q\times p$ matrix
$$ A=
\begin{pmatrix}
a_{0,0} &\cdots & a_{0, p-1}\\
\vdots &\ddots& \vdots\\
a_{q-1, 0} & \cdots&a_{q-1, p-1}\\
\end{pmatrix}
$$
with entries in any algebraic field and such that $a_{0,0} \neq 0.$ We define the transformation $A \mapsto B=T(A),$ where $B$ is the $q\times p$ matrix
\begin{gather}
B=
\begin{pmatrix}
b_{0,0} &\cdots & b_{0, p-1}\\
\vdots &\ddots& \vdots\\
b_{q-1, 0} & \cdots&b_{q-1, p-1}\\
\end{pmatrix}\notag\\
=\frac{1}{a_{0,0}}
\begin{pmatrix}
a_{1,1}a_{0,0}-a_{0,1}a_{1,0} &\cdots& a_{1, p-1}a_{0,0}-a_{0, p-1}a_{1, 0} & a_{1, 0}\\
\vdots&\ddots&\vdots& \vdots\\
a_{q-1, 1}a_{0,0}-a_{0,1}a_{q-1, 0} &\cdots& a_{q-1, p-1}a_{0,0}-a_{0, p-1}a_{q-1, 0} & a_{q-1, 0}\\
-a_{0,1}& \cdots& -a_{0, p-1} & 1
\end{pmatrix}\label{opinv1}
\end{gather}
with entries
\begin{align}
b_{i,j} &  =(a_{i+1,j+1}a_{0,0}-a_{0,j+1}a_{i+1,0})/a_{0,0},\qquad 0\leq i\leq q-2, \quad 0\leq j\leq p-2,\label{26-4-1}\\
b_{q-1,j} &  =-a_{0,j+1}/a_{0,0}, \qquad 0\leq j\leq p-2,\label{26-4-3}\\
b_{i,p-1} & =a_{i+1, 0}/a_{0,0}, \qquad 0\leq i\leq q-2,  \label{26-4-2}\\
b_{q-1,p-1} &  =1/a_{0,0}.\label{26-4-4}
\end{align}
Observe that $b_{q-1,p-1}\neq 0$. This transformation $T$ has an inverse $B\mapsto A=T^{-1}(B)$ given by
\[
\begin{small}
A=\frac{1}{b_{q-1, p-1}}
\begin{pmatrix}
1& -b_{q-1, 0}& \cdots& -b_{q-1, p-2}\\
b_{0, p-1}& b_{0,0} b_{q-1, p-1}-b_{0, p-1}b_{q-1, 0}& \cdots& b_{0, p-2} b_{q-1, p-1}-b_{0, p-1}b_{q-1, p-2}\\
\vdots&\vdots&\ddots&\vdots\\
b_{q-2, p-1}& b_{q-2, 0}b_{q-1, p-1}-b_{q-2, p-1}b_{q-1, 0}&\cdots&b_{q-2, p-2}b_{q-1, p-1}-b_{q-2,p-1}b_{q-1, p-2}\\
\end{pmatrix}.
\end{small}
\]
We will use the notation
$$
A=\frac{\mathbf{1}}{B}
$$
if $B=T(A).$
\par
In this section we obtain a matrix continued fraction expansion for the $q\times p$ matrix  
\begin{equation}\label{def:Fmatrix}
F(z):=\begin{pmatrix}
\phi_{0,0}(z) &\cdots & \phi_{0, p-1}(z)\\
\vdots &\ddots& \vdots\\
\phi_{q-1, 0}(z) & \cdots&\phi_{q-1, p-1}(z)\\
\end{pmatrix}
\end{equation}
with entries 
$$
F_{i,j}(z)=\phi_{i,j}(z), \qquad 0\leq i\leq q-1, \qquad 0\leq j\leq p-1,
$$
as given in \eqref{7-9-22-2}. By Theorem \ref{theorem2}, we can also represent the matrix $F(z)$ in the form
 \begin{equation}\label{def:Fmatrix1}
F(z)=\begin{pmatrix}
A_{0,0}(z) &\cdots & A_{0, p-1}(z)\\
\vdots &\ddots& \vdots\\
A_{q-1, 0}(z) & \cdots&A_{q-1, p-1}(z)\\
\end{pmatrix}.
\end{equation}
We also consider the $q\times p$ matrix  
\begin{equation}\label{def:F1matrix}
F_1(z):=\begin{pmatrix}
\phi^{(1)}_{0,0}(z) &\cdots & \phi^{(1)}_{0, p-1}(z)\\
\vdots &\ddots& \vdots\\
\phi^{(1)}_{q-1, 0}(z) & \cdots&\phi^{(1)}_{q-1, p-1}(z)\\
\end{pmatrix}
\end{equation}
with entries $\phi^{(1)}_{i,j}(z)$, $0\leq i\leq q-1$, $0\leq j\leq p-1$, defined in \eqref{8-26-22-4}.

The following result follows from Theorem \ref{theorem3} and states that the matrix $F(z)$ can be expressed in terms of the matrix $F_1(z)$ using the transformation $T$. It is important to note that on the right-hand side of \eqref{10-21-22-2}, besides the matrix $F_1(z)$ we only use the entries from the first row and first column of the matrix
$H$.
\begin{theorem}\label{10-21-22-1}
The following identity holds between the matrices defined in \eqref{def:Fmatrix} and \eqref{def:F1matrix}:
\begin{equation}\label{10-21-22-2}
F(z)=\frac{\mathbf{1}}{\alpha_{0}(z)+\alpha_{0}^{+}\,F_{1}(z)\, \alpha_{0}^{-}}
\end{equation}
where the matrices $\alpha_{0}(z)$, $\alpha_{0}^{+}$, $\alpha_{0}^{-}$ are defined in \eqref{6-9-23-10}--\eqref{6-9-23-12} (in the case $k=0$).
\end{theorem}
\begin{proof}
First note that $\phi_{0,0}(z)\neq 0$, so we can perform the transformation $T(F)$. For $0\leq i\leq q-2$ and $0\leq j\leq p-2$, by \eqref{26-4-1} and \eqref{6-4-4} we have
\begin{equation}\label{10-21-2-30}
T(F)_{i,j}  =(\phi_{i+1,j+1}(z)\phi_{0,0}(z)-\phi_{0,j+1}(z)\phi_{i+1,0}(z))/\phi_{0,0}(z)=\phi^{(1)}_{i,j}(z).  
\end{equation}
Also, for $0\leq j\leq p-2$, using \eqref{26-4-3} and \eqref{6-4-2}, we can write
\begin{equation}\label{10-21-2-32} 
T(F)_{q-1,j} =-\phi_{0,j+1}(z)/\phi_{0,0}(z)=-\sum_{i=1}^{q} a_{0}^{(i)}\phi_{i-1,j}^{(1)}(z).
\end{equation}
For $0\leq i\leq q-2$, with the aid of \eqref{26-4-2} and \eqref{6-4-3}, we obtain
\begin{equation}\label{10-21-2-31}
T(F)_{i,p-1} =\phi_{i+1, 0}(z)/\phi_{0,0}(z)=\sum_{j=1}^{p} a_{0}^{(-j)}\phi_{i,j-1}^{(1)}(z).
\end{equation}
Finally, by  \eqref{26-4-1} and \eqref{6-4-1},
\begin{equation}\label{10-21-2-33}
T(F)_{q-1,p-1} =1/\phi_{0,0}(z)=z-a_{0}^{(0)}-\sum_{i=1}^{q}\sum_{j=1}^{p}a_{0}^{(i)}a_{0}^{(-j)}\,\phi_{i-1,j-1}^{(1)}(z).
\end{equation}
Therefore, by \eqref{10-21-2-30}--\eqref{10-21-2-33} we obtain
\begin{gather*}
T(F)=\\
\begin{small}
=\begin{pmatrix}
&\phi^{(1)}_{0,0}(z)&\cdots&\phi^{(1)}_{0,p-2}(z) &\sum_{j=1}^{p} a_{0}^{(-j)}\phi_{0,j-1}^{(1)}(z) \\[0.2em]
&\cdots&\cdots&\cdots&\cdots\\[0.1em]
&\phi^{(1)}_{q-2,0}(z)& \cdots&\phi^{(1)}_{q-2,p-2}(z) &\sum_{j=1}^{p} a_{0}^{(-j)}\phi_{q-2,j-1}^{(1)}(z)\\[0.5em]
&- \sum_{i=1}^{q} a_{0}^{(i)}\phi_{i-1,0}^{(1)}(z)&\cdots&-\sum_{i=1}^{q} a_{0}^{(i)}\phi_{i-1,p-2}^{(1)}(z) & z-a_{0}^{(0)}-\sum_{i=1}^{q}\sum_{j=1}^{p}a_{0}^{(i)}a_{0}^{(-j)}\,\phi_{i-1,j-1}^{(1)}(z)
\end{pmatrix}.
\end{small}
\end{gather*}
Then, by the properties of matrix multiplication we get
\begin{equation}\label{10-21-22-41}
T(F)=\alpha_{0}(z)+\alpha_{0}^{+}\,F_{1}(z)\,\alpha_{0}^{-}.
\end{equation}
The detailed verification of \eqref{10-21-22-41} is left to the reader.
\end{proof}

With the help of the matrices defined in \eqref{6-9-23-10}--\eqref{6-9-23-12}, we introduce the following transformations for matrices $X$ of size $q\times p$:
\begin{equation}\label{FLTalphas}
\tau_{\alpha,k}(X)=\frac{\mathbf{1}}{\alpha_{k}(z)+\alpha_{k}^{+}\,X\,\alpha_{k}^{-}},\qquad k\geq 0.
\end{equation}
  
As in Section~\ref{sec:lattpathoper}, for any integer $k\geq 1$, let $\mathcal{H}_{k}$ be the operator (possibly unbounded) defined by 
\[
\begin{cases}
\mathcal{H}_{k}\,e_{0}=\sum_{m=0}^{p} a_{k}^{(-m)} e_{m},\\[0.3em]
\mathcal{H}_{k}\, e_{n}=\sum_{m=1}^{n} a_{n-m+k}^{(m)}\,e_{n-m}+\sum_{m=0}^{p} a_{n+k}^{(-m)}\,e_{n+m},\quad 0< n< q,\\[0.3em]
\mathcal{H}_{k}\, e_{n}=\sum_{m=1}^{q} a_{n-m+k}^{(m)}\,e_{n-m}+\sum_{m=0}^{p} a_{n+k}^{(-m)}\,e_{n+m},\quad n\geq q,
\end{cases}
\]
and extended by linearity to $\mathcal{E}_{0}=\mbox{span}\{e_{n}\}_{n=0}^{\infty}$. In the basis $\{e_{n}\}_{n=0}^{\infty}$, the matrix representation of the operator $\mathcal{H}_{k}$  is the banded matrix
\begin{equation}\label{def:Hk}
H^{[k]}=\begin{pmatrix}
a_{k}^{(0)} & \ldots & a_{k}^{(q)} & & & \\
\vdots & a_{k+1}^{(0)} & \ldots & a_{k+1}^{(q)} & & \\
a_{k}^{(-p)} & \vdots & a_{k+2}^{(0)} & \ldots & a_{k+2}^{(q)} & \\
 & a_{k+1}^{(-p)} & \vdots & \ddots & & \ddots \\
 &  & a_{k+2}^{(-p)} & & \ddots & \\
 & & & \ddots & 
 \end{pmatrix}.
\end{equation}
Observe that $H^{[k]}$ is the infinite  matrix obtained by removing the first $k$ rows and the first $k$ columns of the matrix $H$. The system of resolvent functions $\phi_{i,j}^{(k)}(z)$ associated with the operator  $\mathcal{H}_{k}$ is given by the following formulas
\[
\phi_{i,j}^{(k)}(z):=\sum_{n=0}^{\infty}\frac{\langle\mathcal{H}_{k}^{n}\, e_{j},e_{i}\rangle}{z^{n+1}}\qquad i,j \geq 0.
\]

As above in this section, we define the $q\times p$ matrix
\begin{equation}\label{def:Fkmatrix}
F_k(z)=\begin{pmatrix}
\phi^{(k)}_{0,0}(z) &\cdots & \phi^{(k)}_{0, p-1}(z)\\
\vdots&\ddots&\vdots\\
\phi^{(k)}_{q-1, 0}(z) & \cdots&\phi^{(k)}_{q-1, p-1}(z)
\end{pmatrix},\qquad k\geq 1,
\end{equation}
with entries $\phi^{(k)}_{i,j}(z)$, $0\leq i\leq q-1$, $0\leq j\leq p-1$. Observe that $\phi_{0,0}^{(k)}(z)\neq 0$ for all $k\geq 1$.
 
The following is one of the main results of the paper, and states that the matrix $F(z)$ can be expressed as a matrix continued fraction.
\begin{theorem}\label{TH6-8-23-1}
For any $k\geq 1$, the following identity holds between the matrices $F(z)$ and $F_{k}(z)$ defined in \eqref{def:Fmatrix}--\eqref{def:Fmatrix1} and \eqref{def:Fkmatrix}:
\begin{equation}\label{finiteMCFforF}
F(z)=(\tau_{\alpha,0}\circ\tau_{\alpha,1}\circ\cdots\circ\tau_{\alpha,k-1})(F_{k}(z)).
\end{equation}
As a result, we have the following formal expansion
\begin{equation}\label{def:FCFmatrix}
F(z)=\cfrac{{\bf 1}}{\alpha_{0}(z)+\alpha^{+}_0\cfrac{{\bf 1}}{\alpha_{1}(z)+\alpha^{+}_1\cfrac{{\bf 1}}{\alpha_{2}(z)+\alpha^{+}_2\cfrac{{\bf 1}}
{\alpha_{3}(z)+\ddots}\,\alpha^{-}_2}\,\alpha^{-}_1}\,\alpha^{-}_0},
\end{equation}
where the matrix coefficients in \eqref{def:FCFmatrix} are defined in \eqref{6-9-23-10}--\eqref{6-9-23-12}.
\end{theorem}
\begin{proof}
If we replace the matrix $H$ by $H^{[k]}$ and apply Theorem~\ref{10-21-22-1}, we immediately get the relation
\begin{equation}\label{10-22-22-2}
F_k(z)=\frac{{\bf 1}}{\alpha_{k}(z)+\alpha^{+}_k\,F_{k+1}(z)\,\alpha^{-}_k}=\tau_{\alpha,k}(F_{k+1}(z)),\qquad k\geq 1.
\end{equation}
Combining \eqref{10-21-22-2} and an iteration of \eqref{10-22-22-2}, we obtain \eqref{finiteMCFforF}. 
\end{proof}

\begin{remark}\label{R6-12-23-1}
It is worth emphasizing the approach to obtain the matrix continued fraction for $F(z)$ directly from the banded matrix $H$ in \eqref{matrixH}. In the expansion \eqref{def:FCFmatrix}, the matrices $\alpha_{k}(z)$, $\alpha^{+}_k$, $\alpha^{-}_k$ defined in \eqref{6-9-23-12}, \eqref{6-9-23-10}, \eqref{6-9-23-11} are constructed by selecting specific entries of $H$. The entries selected are the $(k+1)$-st entry on the main diagonal (for $\alpha_{k}(z)$), the entries from the $(k+1)$-st row located to the right of the main diagonal (for $\alpha_{k}^{+}$), and the entries from the $(k+1)$-st column located below the main diagonal (for $\alpha_{k}^{-}$).
\end{remark}

The next essential result of the paper concerns the construction of a matrix continued fraction for a matrix built of formal Laurent series $W_{i,j}(z)$ associated with the collections $\mathcal{P}_{[n,i,j]}$ of lattice paths. Consider now the $q\times p$ matrix
\begin{equation}\label{6-9-23-21}
G(z)=\begin{pmatrix}
W_{0,0}(z) &\cdots & W_{0, p-1}(z)\\
\vdots &\ddots& \vdots\\
W_{q-1, 0}(z) & \cdots&W_{q-1, p-1}(z)\\
\end{pmatrix},
\end{equation}
with entries defined in \eqref{7-1-22-8}. Observe that according to Theorem \ref{10-3-22-1}, we can also represent $G(z)$ as
\[
G(z)=\begin{pmatrix}
\zeta_{0,0}(z) &\cdots & \zeta_{0, p-1}(z)\\
\vdots &\ddots& \vdots\\
\zeta_{q-1, 0}(z) & \cdots & \zeta_{q-1, p-1}(z)\\
\end{pmatrix}
\]
with entries defined by \eqref{10-05-22-11}.

Let us introduce the matrices that will be used to construct the matrix continued fraction for $G(z)$. The reader should keep in mind that these matrices are defined as the $\alpha$ matrices in \eqref{6-9-23-10}--\eqref{6-9-23-12} but using the entries of the matrix $K$ in \eqref{6-12-23-31} instead of the matrix $H$. First we define the matrices obtained from the diagonal entries of $K$. For any integer $k \geq \min (p,q)$, the $q\times p$ matrix $\beta_{k}(z)$ is defined as follows: 
\begin{equation}\label{6-9-23-40}
\beta_{k}(z)=\alpha_{k}(z)=\begin{pmatrix}
0 &\cdots & 0 & 0\\
\vdots &\ddots& \vdots & \vdots\\
0 & \cdots & 0 & 0 \\
0 & \cdots & 0 & z-a^{(0)}_k
\end{pmatrix}.
\end{equation}
For any integer $0 \leq k \leq \min (p,q)-1$, the $q\times p$ matrix $\beta_{k}(z)$ is 
\begin{equation}\label{6-9-23-41}
\beta_{k}(z)=\begin{pmatrix}
0 &\cdots & 0 & 0\\
\vdots &\ddots& \vdots & \vdots\\
0 & \cdots & 0 & 0 \\
0 & \cdots & 0 & z-d^{(0)}_k
\end{pmatrix}
\end{equation}
where the coefficients $d^{(0)}_{k}$ in the above matrix are defined in the following way:
\begin{equation}\label{6-9-23-51}
d^{(0)}_{k}=a^{(0)}_{k}+\sum ^{p-k}_{l=1}\sum^{q-k}_{m=1}a^{(-(l+k))}_{-l}a^{(m+k)}_{-m}V_{-l,-m}(z).
\end{equation}
Now we define the matrices obtained from the off-diagonal entries of $K$. For $ k \geq \min (p,q)$, the $q\times q$ matrix $\beta^{+}_k$ is defined as
\begin{equation}\label{6-9-23-42}
\beta^{+}_k=\alpha^{+}_k=\begin{pmatrix}
1 & 0 & \cdots & 0 & 0\\
0 & 1 & \cdots & 0 & 0\\
\vdots & \vdots & \ddots & \vdots & \vdots \\
0 & 0 & \cdots & 1 & 0\\
-a^{(1)}_k & -a_{k}^{(2)} & \cdots & -a^{(q-1)}_k & -a_{k}^{(q)}
\end{pmatrix}.
\end{equation}
For $0\leq k \leq \min (p,q)-1$, the $q\times q$ matrix $\beta^{+}_k$ is given by
\begin{equation}\label{6-9-23-43}
\beta^{+}_k=\begin{pmatrix}
1 & 0 & \cdots & 0 & 0\\
0 & 1 & \cdots & 0 & 0\\
\vdots & \vdots & \ddots & \vdots & \vdots \\
0 & 0 & \cdots & 1 & 0\\
-d^{(1)}_k & -d_{k}^{(2)} & \cdots & -d^{(q-1)}_k & -d_{k}^{(q)}
\end{pmatrix}
\end{equation}
where the coefficients $d^{(i)}_{k}$ in the above matrix are defined in the following way:
\begin{equation}\label{6-9-23-52}
\begin{cases}
d^{(i)}_{k}=a^{(i)}_{k}+\sum ^{p-k}_{l=1}\sum^{q-(k+i)}_{m=1}a^{(-(l+k))}_{-l}a^{(m+k+i)}_{-m}V_{-l,-m}(z), & 
1\leq i\leq q-k-1, \\[0.3em]
d_{k}^{(i)} =a^{(i)}_{k}, & q-k\leq i\leq q.\\[0.3em]
\end{cases}
\end{equation}
Additionally, for $ k \geq \min (p,q)$, the $p\times p$ matrix $\beta^{-}_k$  is 
\begin{equation}\label{6-9-23-44}
\beta^{-}_k=\alpha^{-}_k=\begin{pmatrix}
1 & 0 & \cdots & 0 & a_{k}^{(-1)}\\[0.4em]
0 & 1 & \cdots & 0 & a_{k}^{(-2)}\\
\vdots & \vdots & \ddots & \vdots & \vdots \\
0 & 0 & \cdots & 1 & a_{k}^{(-p+1)}\\[0.4em]
0 & 0 & \cdots & 0 & a_{k}^{(-p)}
\end{pmatrix}. 
\end{equation}
For $0\leq k \leq \min (p,q)-1$, the $p\times p$ matrix $\beta^{-}_k$ is given by:
\begin{equation}\label{6-9-23-45}
\beta^{-}_k=\begin{pmatrix}
1 & 0 & \cdots & 0 & d_{k}^{(-1)}\\[0.4em]
0 & 1 & \cdots & 0 & d_{k}^{(-2)}\\
\vdots & \vdots & \ddots & \vdots & \vdots \\
0 & 0 & \cdots & 1 & d_{k}^{(-p+1)}\\[0.4em]
0 & 0 & \cdots & 0 & d_{k}^{(-p)}
\end{pmatrix}.
\end{equation}
The coefficients $d^{(-j)}_{k}$ in the above matrix are defined according to the following formulas
\begin{equation}\label{6-9-23-53}
\begin{cases}
d^{(-j)}_{k}=a^{(-j)}_{k}+\sum ^{p-(k+j)}_{l=1}\sum^{q-k}_{m=1}a^{(-(l+k+j))}_{-l}a^{(m+k)}_{-m}V_{-l,-m}(z), & 
1\leq j\leq p-k-1, \\[0.3em]
d_{k}^{(-j)} =a^{(-j)}_{k}, & p-k\leq j\leq p.\\[0.3em]
\end{cases}
\end{equation}
The transformations that will be involved in the matrix continued fraction for $G(z)$ are defined by
\[
\tau_{\beta,k}(X)=\frac{\mathbf{1}}{\beta_{k}(z)+\beta_{k}^{+}\,X\,\beta_{k}^{-}},\qquad k\geq 0,
\]
for matrices $X$ of size $q\times p$.

We need now to introduce some series associated with the banded matrix $K=(K_{i,j})_{i,j=0}^{\infty}$ given in \eqref{6-12-23-31}. 
We denote by $K_{r}$, $r\geq 1$, the submatrix of $K$ obtained by deleting the first $r$ rows and the first $r$ columns of $K$. Also set $K_{0}=K$. Note that in particular we have $K_{r}=H^{[r]}$ for all $r\geq \min(p,q)$, where $H^{[r]}$ is defined in \eqref{def:Hk}.
It is evident that the matrix $K_{r}=((K_{r})_{i,j})_{i,j\geq 0}$ is the matrix representation of a formal Laurent series $\mathcal{K}_{r}(z)\in\mathcal{L}_{0}((z^{-1}))$ (see properties $1)$--$2)$ in Section~\ref{sec:rfH}). Additionally, the series $zI-\mathcal{K}_{r}(z)$ is invertible in $\mathcal{L}_{0}((z^{-1}))$ according to Lemma~\ref{lemmainversion}. Consequently, the matrix $zI-K_{r}$ is invertible, and we define the scalar series $\zeta^{(r)}_{i,j}(z)$ as the $(i,j)$-entry of the matrix $(zI-K_{r})^{-1}$. In other words, we can express this as:
\begin{equation}\label{6-29-23-1}
(zI-K_{r})^{-1}=(\zeta^{(r)}_{i,j}(z))_{i,j\geq 0}.
\end{equation}
We also define the matrices
\begin{equation}\label{def:Grmatrix}
G_{r}(z):=\begin{pmatrix}
\zeta^{(r)}_{0,0}(z) &\cdots & \zeta^{(r)}_{0, p-1}(z)\\
\vdots&\ddots&\vdots\\
\zeta^{(r)}_{q-1,0}(z) & \cdots&\zeta^{(r)}_{q-1, p-1}(z)
\end{pmatrix},\qquad r\geq 1.
\end{equation}
Furthermore, we define
\begin{equation}\label{idG0G}
G_{0}(z)=G(z).
\end{equation}
The result presented below, which constitutes one of the main contributions of the paper, establishes that the matrix $G(z)$ can be expressed as a matrix continued fraction.

\begin{theorem}\label{TH6-9-23-1}
For any $k\geq 1$, the following identity holds between the matrices $G(z)$ and $G_{k}(z)$ defined in \eqref{6-9-23-21} and \eqref{def:Grmatrix}: 
\begin{equation}\label{relGGk}
G(z)=(\tau_{\beta,0}\circ\tau_{\beta,1}\circ\cdots\circ\tau_{\beta,k-1})(G_{k}(z)).
\end{equation}
Thus, we have the formal expansion
\begin{equation}\label{6-3-23-1}
G(z)=\cfrac{{\bf 1}}{\beta_{0}(z)+\beta^{+}_0\cfrac{{\bf 1}}{\beta_{1}(z)+\beta^{+}_1\cfrac{{\bf 1}}{\beta_{2}(z)+\beta^{+}_2\cfrac{{\bf 1}}
{\beta_{3}(z)+\ddots}\,\beta^{-}_2}\,\beta^{-}_1}\,\beta^{-}_0}, 
\end{equation}
where the matrix coefficients in \eqref{6-3-23-1} are defined in 
\eqref{6-9-23-40}--\eqref{6-9-23-53}.
\end{theorem}
\begin{proof}
We claim that for any $r \geq 0$, the following relations hold between the series defined  in \eqref{10-05-22-11} and \eqref{6-29-23-1}: 
\begin{align}
\zeta_{0,0}^{(r)}(z) & =\frac{1}{z-d_{r}^{(0)}-\sum_{i=1}^{q}\sum_{j=1}^{p}d_{r}^{(i)}d_{r}^{(-j)}\,\zeta_{i-1,j-1}^{(r+1)}(z)},
\label{V6-4-1}\\
\zeta_{0,j}^{(r)}(z) & =\zeta_{0,0}^{(r)}(z)\sum_{i=1}^{q} d_{r}^{(i)}\zeta_{i-1,j-1}^{(r+1)}(z), \qquad j\geq 1,\label{V6-4-2}\\
\zeta_{i,0}^{(r)}(z) & =\zeta_{0,0}^{(r)}(z)\sum_{j=1}^{p} d_{r}^{(-j)}\zeta_{i-1,j-1}^{(r+1)}(z),\qquad i\geq 1,\label{V6-4-3}\\
\zeta_{i,j}^{(r)}(z) & =\frac{\zeta_{i,0}^{(r)}(z)\,\zeta_{0,j}^{(r)}(z)}{\zeta_{0,0}^{(r)}(z)}+\zeta^{(r+1)}_{i-1,j-1}(z),\qquad i,j\geq 1,\label{V6-4-4}
\end{align}
where for $r\geq \min(p,q)$ we use the definition $d_{r}^{(i)}=a_{r}^{(i)}$, $-p\leq i\leq q$, and for $0\leq r\leq \min(p,q)-1$, the coefficients $d_{r}^{(i)}$, $-p\leq i\leq q$, are defined as in \eqref{6-9-23-51}, \eqref{6-9-23-52}, and \eqref{6-9-23-53} (with $k$ replaced by $r$ in those formulas). Equivalently, for all $r\geq 0$ we have $d_{r}^{(i)}=(K_{r})_{0,i}$ for $0\leq i\leq q$, and $d_{r}^{(-j)}=(K_{r})_{j,0}$ for $1\leq j\leq p$, i.e.,
\[
K_{r}=\begin{pmatrix}
d_{r}^{(0)} & \cdots & d_{r}^{(q)} & & & \\
\vdots & d_{r+1}^{(0)} & \cdots & d_{r+1}^{(q)} & &\\
d_{r}^{(-p)} & \vdots & d_{r+2}^{(0)} & \cdots & d_{r+2}^{(q)} &\\
 & d_{r+1}^{(-p)} & \vdots & \ddots & & \ddots \\
 &  & d_{r+2}^{(-p)} & & \ddots \\
 & & & \ddots &  
\end{pmatrix}.
\]

The strategy we employ in proving the formulas \eqref{V6-4-1}--\eqref{V6-4-4} is the same as the one used in the proof of Theorem \ref{theorem3}, with the substitution of the matrix $H$ by the matrix $K_r$. Initially, we partition the matrix $zI-K_r$ into blocks as follows:
\[
zI-K_r=\begin{pmatrix}
A & B \\
C & D 
\end{pmatrix},
\]
where $A$ is the $1\times1$ matrix
\[
A=\begin{pmatrix}
z-d_r^{(0)}
\end{pmatrix}.
\]
So the block $B$ is the infinite row vector   
\[
B=\begin{pmatrix}
-d_{r}^{(1)} & \cdots & -d_{r}^{(q)} & 0 & 0 & 0& \cdots\\
\end{pmatrix}
\]
with entries
\[
\begin{cases}
b_{j}=-d_{r}^{(j)}, & 1\leq j\leq q,\\
b_{j}=0, & \mbox{otherwise}.
\end{cases}
\]
$C$ is the infinite column vector
\[
C=\begin{pmatrix}
-d_{r}^{(-1)}\,\\
\vdots\,\\
-d_{r}^{(-p)}\,\\
 0\,\\
0\,\\
0\,\\
\vdots\\
\end{pmatrix}
\]
with entries
\[
\begin{cases}
c_{i}=-d_{r}^{(-i)}, & 1\leq i\leq p,\\
c_{i}=0, & \mbox{otherwise}.
\end{cases}
\]
The block $D=zI-{K}_{r+1}$ is the matrix
\[
\begin{pmatrix}
z-d_{r+1}^{(0)} & \cdots & -d_{r+1}^{(q)} & & & \\
\vdots & z-d_{r+2}^{(0)} & \cdots & -d_{r+2}^{(q)} & &\\
-d_{r+1}^{(-p)} & \vdots & z-d_{r+3}^{(0)} & \cdots & -d_{r+3}^{(q)} &\\
 & -d_{r+2}^{(-p)} & \vdots & \ddots & & \ddots \\
 &  & -d_{r+3}^{(-p)} & & \ddots \\
 & & & \ddots &  
\end{pmatrix}.
\]
We note that the matrix $D$ is invertible and 
\[
D^{-1}=(\zeta_{i-1,j-1}^{(r+1)}(z))_{i,j=1}^{\infty}.
\]
The key step in the proof of the formulas \eqref{V6-4-1}--\eqref{V6-4-4} is the application of the following identity:
\begin{gather}
(zI-K_{r})^{-1}=
\begin{pmatrix}
A & B \\
C & D 
\end{pmatrix}^{-1}
=\begin{pmatrix}
(A-BD^{-1}C)^{-1}& -(A-BD^{-1}C)^{-1}BD^{-1}\\[0.1em]
-D^{-1}C(A-BD^{-1}C)^{-1}& D^{-1}C(A-BD^{-1}C)^{-1}BD^{-1}+D^{-1}
\end{pmatrix}.\label{6-30-23-1}
\end{gather}
A careful analysis of the proof of Theorem \ref{theorem3}  reveals that the identity \eqref{6-30-23-1} is applicable in this case. By employing the same arguments as in Theorem \ref{theorem3}, we can establish the validity of the formulas \eqref{V6-4-1}--\eqref{V6-4-4}.
Therefore, from \eqref{V6-4-1}--\eqref{V6-4-4} we can deduce that
\[
G_{r}(z)=\frac{\mathbf{1}}{\beta_{r}(z)+\beta_{r}^{+}\,G_{r+1}(z)\,\beta_{r}^{-}},\qquad\mbox{for all}\,\,r\geq 0.
\]
These relations, along with \eqref{idG0G}, imply \eqref{relGGk}.
\end{proof}

\subsection{Matrix continued fraction associated with the collections of paths $\widehat{\mathcal{D}}_{[n,i,j]}$}

To provide a more comprehensive exposition of the applicability of the algorithm to construct matrix continued fractions described in Remark~\ref{R6-12-23-1}, we present an additional result. We show that a certain matrix with entries given by the series $V_{i,j}(z)$, $i,j\leq -1$, defined in \eqref{def:Vijseries} and associated with the collections
 $\widehat{\mathcal{D}}_{[n,i,j]}$, can be represented as a matrix continued fraction. As described in 
Section \ref{Wmatrix}, 
the collection $\widehat{\mathcal{D}}_{[n,i,j]}$, $i,j\leq -1$, consists of those paths in $\mathcal{P}_{[n,i,j]}$ constrained to remain below or touch the line $y=1$. The matrix in question is
\begin{equation}\label{U6-10-23-1}
V(z):=\begin{pmatrix}
V_{-1,-1}(z) &\cdots & V_{-p, -1}(z)\\
\vdots &\ddots& \vdots\\
V_{-1, -q}(z) & \cdots& V_{-p, -q}(z)\\
\end{pmatrix}.
\end{equation}
To construct the matrix continued fraction for $V(z)$, we can use the matrix $E$ defined in \eqref{matrixE} (see discussion in subsection \ref{subsecDhat}), the same way the continued fraction for $F(z)$ is constructed using the matrix $H$. The ingredients for the construction are the following.

Let $\mathcal{V}=\mathcal{V}_{0}$ be the operator on $\mathcal{E}_{0}=\mbox{span}\{e_{n}\}_{n=0}^{\infty}$ with matrix representation in the basis $\{e_{n}\}_{n=0}^{\infty}$ given by the matrix $E$ in \eqref{matrixE}, that is
\[
\begin{cases}
\mathcal{V} e_{0}=\sum_{m=0}^{p} a_{-(m+1)}^{(-m)} e_{m},\\[0.3em]
\mathcal{V} e_{n}=\sum_{m=1}^{n} a_{-(n+1)}^{(m)}\,e_{n-m}+\sum_{m=0}^{p} a_{-(n+m+1)}^{(-m)}\,e_{n+m},\quad 0< n< q,\\[0.3em]
\mathcal{V} e_{n}=\sum_{m=1}^{q} a_{-(n+1)}^{(m)}\,e_{n-m}+\sum_{m=0}^{p} a_{-(n+m+1)}^{(-m)}\,e_{n+m},\quad n\geq q.
\end{cases}
\]
Applying \eqref{6-26-22-1} and taking into account the discussion at the end of subsection \ref{subsecDhat}, we can characterize the power series $V_{-(j+1),-(i+1)}(z)$ defined in \eqref{def:Vijseries} as resolvent functions of the operator $\mathcal{V}$ as follows:
\[
V_{-(j+1),-(i+1)}(z)=\sum_{n=0}^{\infty}\frac{\langle\mathcal{V}^{n} e_{j},e_{i}\rangle}{z^{n+1}},\qquad i,j\geq 0.
\]
For each $k\geq 1$, let $E^{[k]}$ denote the submatrix of $E$ obtained after deleting the first $k$ rows and columns of $E$, and let $\mathcal{V}_{k}$ be the operator on $\mathcal{E}_{0}$ with matrix representation $E^{[k]}$. Then we define the formal series
\[
V_{-(j+1),-(i+1)}^{(k)}(z):=\sum_{n=0}^{\infty}\frac{\langle\mathcal{V}_{k}^{n}\, e_{j},e_{i}\rangle}{z^{n+1}},\qquad i,j\geq 0,\quad k\geq 1,
\] 
and the matrices
\begin{equation}\label{def:Vkmatrix}
V_{k}(z):=\begin{pmatrix}
V_{-1,-1}^{(k)}(z) &\cdots & V_{-p, -1}^{(k)}(z)\\
\vdots &\ddots& \vdots\\
V_{-1, -q}^{(k)}(z) & \cdots& V_{-p, -q}^{(k)}(z)\\
\end{pmatrix},\qquad k\geq 1.
\end{equation}
A straightforward application of Theorem \ref{TH6-8-23-1} (or the algorithm described in Remark~\ref{R6-12-23-1}) gives the following result:

\begin{proposition}\label{P6-10-23-1}
For any $k\geq 1$, the following identity holds between the matrices $V(z)$ and $V_{k}(z)$ defined in \eqref{U6-10-23-1} and \eqref{def:Vkmatrix}:
\[
V(z)=(\tau_{\nu,0}\circ\tau_{\nu,1}\circ\cdots\circ\tau_{\nu,k-1})(V_{k}(z)),
\] 
where we use the transformations defined as in \eqref{FLTalphas} but corresponding to the $\nu$ matrices given below. Thus, formally we have:
\begin{equation}\label{Umatrix}
V(z)=\cfrac{{\bf 1}}{\nu_{0}(z)+\nu^{+}_0\cfrac{{\bf 1}}{\nu_{1}(z)+\nu^{+}_1\cfrac{{\bf 1}}{\nu_{2}(z)+\nu^{+}_2\cfrac{{\bf 1}}
{\nu_{3}(z)+\ddots}\,\nu^{-}_2}\,\nu^{-}_1}\,\nu^{-}_0}. 
\end{equation}
The coefficient matrices used in \eqref{Umatrix} are defined as follows:
For any integer $k\geq 0$,  let $\nu^{+}_k$ be the $q\times q$  matrix
$$
\nu^{+}_k=\begin{pmatrix}
1 & 0 & \cdots & 0 & 0\\
0 & 1 & \cdots & 0 & 0\\
\vdots & \vdots & \ddots & \vdots & \vdots \\
0 & 0 & \cdots & 1 & 0\\
-a^{(1)}_{-(k+2)}& -a_{-(k+3)}^{(2)} & \cdots & -a^{(q-1)}_{-(q+k)} & -a_{-(q+k+1)}^{(q)}
\end{pmatrix}
$$
and let $\nu^{-}_k$ be the $p\times p$  matrix
$$
\nu^{-}_k=\begin{pmatrix}
1 & 0 & \cdots & 0 & a_{-(k+2)}^{(-1)}\\[0.4em]
0 & 1 & \cdots & 0 & a_{-(k+3)}^{(-2)}\\
\vdots & \vdots & \ddots & \vdots & \vdots \\
0 & 0 & \cdots & 1 & a_{-(p+k)}^{(-p+1)}\\[0.4em]
0 & 0 & \cdots & 0 & a_{-(p+k+1)}^{(-p)}
\end{pmatrix}.
$$
Also, let  $\nu_{k}(z)$ be the $q\times p$  matrix
$$
\nu_{k}(z)=\begin{pmatrix}
0 &\cdots & 0 & 0\\
\vdots &\ddots& \vdots & \vdots\\
0 & \cdots & 0 & 0 \\
0 & \cdots & 0 & z-a^{(0)}_{-(k+1)}
\end{pmatrix}.
$$
\end{proposition}

\subsection{Scalar continued fraction in the case $p=q=1$}\label{subsecscf}

We will discuss now the case $p=q=1$ and present continued fractions for $A_{0,0}(z)$ and $W_{0,0}(z)$. Note that in this case the matrices \eqref{def:Fmatrix} and \eqref{6-9-23-21} are scalars and we have $F(z)=A_{0,0}(z)$ and $G(z)=W_{0,0}(z)$. Consider three arbitrary bi-infinite sequences of complex numbers $(a^{(k)}_{n})_{n \in \mathbb{Z}}$, $-1\leq k\leq 1$, and  construct the one-sided matrix
\[
H=\begin{pmatrix}
a_{0}^{(0)} & a_{0}^{(1)} & & & \\[0.3em]
a_{0}^{(-1)} & a_{1}^{(0)} & a_{1}^{(1)} & \\[0.3em]
 & a_{1}^{(-1)} & a_{2}^{(0)} & a_{2}^{(1)} &  \\
 &  &  \ddots & \ddots & \ddots    
\end{pmatrix}
\]
and the two-sided matrix
\[
W=\begin{pmatrix}
& \ddots &  \ddots & \ddots & &  & & & \\
& & a_{-2}^{(-1)}& a_{-1}^{(0)} & a_{-1}^{(1)} & & & &\\[0.3em]
& & & a_{-1}^{(-1)} & a_{0}^{(0)} & a_{0}^{(1)} & &\\[0.3em]
& & &  & a_{0}^{(-1)} & a_{1}^{(0)} & a_{1}^{(1)} & & \\
& & & &  &  \ddots & \ddots & \ddots &   
\end{pmatrix}.
\]
Applying Theorem \ref{TH6-8-23-1}, the continued fraction for $A_{0,0}(z)$ obtained from the entries of $H$ is
\[
A_{0,0}(z)=\cfrac{1}{z-a_{0}^{(0)}-\cfrac{a_{0}^{(-1)}\,a_{0}^{(1)}}{z-a_{1}^{(0)}-\cfrac{a_{1}^{(-1)}\,a_{1}^{(1)}}{z-a_{2}^{(0)}-\cfrac{a_{2}^{(-1)}\,a_{2}^{(1)}}{\ddots}}}}.
\]
Furthermore, note that in this case the matrix \eqref{U6-10-23-1} reduces to $V_{-1,-1}(z)$ and so by Proposition \ref{P6-10-23-1} we have
\begin{equation}\label{6-17-23-1}
V_{-1,-1}(z)=\cfrac{1}{z-a_{-1}^{(0)}-\cfrac{a_{-2}^{(-1)}\,a_{-2}^{(1)}}{z-a_{-2}^{(0)}-\cfrac{a_{-3}^{(-1)}\,a_{-3}^{(1)}}{z-a_{-3}^{(0)}-\cfrac{a_{-4}^{(-1)}\,a_{-4}^{(1)}}{\ddots}}}}.
\end{equation}
To obtain the continued fraction for $W_{0,0}(z)$, we first construct the matrix $K$ using formulas \eqref{6-12-23-31}.
The entries in the banded matrix $K$ are identical to those in the matrix $H$ except for 
the entry $K_{0,0}$ in the first row and first column. By \eqref{6-12-23-31}, the formula for the entry $K_{0,0}$ is given by
$$
K_{0,0} =a_{0}^{(0)}+a^{(-1)}_{-1}a^{(1)}_{-1}\,V_{-1,-1}(z).
$$
Therefore, the matrix $K$ is an infinite tridiagonal matrix
\[
K=\begin{pmatrix}
a_{0}^{(0)}+a^{(-1)}_{-1}a_{-1}^{(1)}\,V_{-1,-1}(z) & a_{0}^{(1)} & & & \\[0.3em]
a_{0}^{(-1)} & a_{1}^{(0)} & a_{1}^{(1)} & \\[0.3em]
 & a_{1}^{(-1)} & a_{2}^{(0)} & a_{2}^{(1)} &  \\
 &  &  \ddots & \ddots & \ddots    
\end{pmatrix}.
\]
With the help of Theorem \ref{TH6-9-23-1}, Remark \ref{R6-12-23-1}, and \eqref{6-17-23-1}, we obtain the following continued fraction for $W_{0,0}(z)$:
\begin{equation}\label{scalardcf}
W_{0,0}(z)=\cfrac{1}{z-a_{0}^{(0)}-\cfrac{a^{(-1)}_{-1}\,a_{-1}^{(1)}}{z-a_{-1}^{(0)}-\cfrac{a_{-2}^{(-1)}\,a_{-2}^{(1)}}{z-a_{-2}^{(0)}-\cfrac{a_{-3}^{(-1)}\,a_{-3}^{(1)}}{\ddots}}}-\cfrac{a_{0}^{(-1)}\,a_{0}^{(1)}}{z-a_{1}^{(0)}-\cfrac{a_{1}^{(-1)}\,a_{1}^{(1)}}{z-a_{2}^{(0)}-\cfrac{a_{2}^{(-1)}\,a_{2}^{(1)}}{\ddots}}}}.
\end{equation}

\section{Lattice paths, resolvents of  the principal truncations of $H$, and rational approximation}\label{12-12-22-1}

Let  $n\geq 1$ and $\ell\geq 0$ be integers. For $0 \leq i,j\leq n-1 $, we denote by $\mathcal{D}_{[\ell, i, j, n]}$ the collection of all lattice paths of 
length $\ell$, with initial
point $(0,i)$, final point $(\ell,j)$, that have no vertex below the $x$-axis or above the line $y=n-1$. The weight polynomial associated with the collection $\mathcal{D}_{[\ell,i,j,n]}$ is denoted
\[
A_{[\ell,i,j,n]} :=\sum_{\gamma\in\mathcal{D}_{[\ell,i,j,n]}}w(\gamma), \qquad n \geq 1, \quad  \ell \geq 0,  \quad  0\leq i,j \leq n-1.
\]
Recall that if $\mathcal{D}_{[\ell,i,j,n]}=\emptyset$, then by definition $A_{[\ell,i,j,n]}=0$. We introduce now  the formal power series generated by the sequences of weight polynomials $A_{[\ell,i,j,n]}.$
For  $n  \geq 1$ and $0 \leq i,j\leq n-1$, let
\begin{equation}\label{12-12-22-3}
A_{i,j,n}(z) :=\sum_{\ell=0}^{\infty}\frac{A_{[\ell,i,j,n]}}{z^{\ell+1}}.
\end{equation}
\par
Let $n$ be a positive integer. Denote by $H_{n}$ the principal $n\times n$ truncation of the banded matrix $H$, that is
\begin{equation}\label{10-26-22-1}
H_{n}=\begin{pmatrix}
a_{0}^{(0)} & a_{0}^{(1)} &\cdots &  a_{0}^{(q)}&  & 0 \\
a_{0}^{(-1)} & \ddots & \ddots&&\ddots \\
\vdots & \ddots & \ddots & \ddots &&a_{n-q-1}^{(q)}\\
a_{0}^{(-p)} & & \ddots  & \ddots & \ddots& \vdots\\
 & \ddots & & \ddots & \ddots & a_{n-2}^{(1)}\\[0.3em]
0 &  & a_{n-p-1}^{(-p)} & \cdots & a_{n-2}^{(-1)} & a_{n-1}^{(0)}
\end{pmatrix}.
\end{equation}
Let $Q_{n}(z):=\det(z I_{n}-H_{n})$ be the characteristic polynomial of the matrix $H_{n}$, which is a monic polynomial of degree $n$. For $0\leq i,j \leq n-1$ we introduce the functions
\begin{align}
 R_{i,j,n}(z) & :=\langle (zI_n-H_{n})^{-1} e_{j}, e_{i}\rangle =\sum_{\ell=0}^{\infty}\frac{\langle{H}_{n}^{\ell}\, e_{j},e_{i}\rangle}{z^{\ell+1}},\label{12-12-22-4}\\
P_{i,j,n}(z) & :=Q_n(z)\langle (zI_{n}-H_{n})^{-1} e_{j}, e_{i}\rangle=\langle Q_n(z)(zI_{n}-H_{n})^{-1} e_{j}, e_{i}\rangle.\notag 
\end{align}
Here $\{e_{i}\}_{i=0}^{n-1}$ denotes the standard basis in $\mathbb{C}^{n}$. Since $Q_{n}(z)=\det(zI_n-H_n)$, it is clear that $P_{i,j,n}(z)$ is a polynomial in $z$ with $\deg P_{i,j,n}\leq n-1$, and so
\[
R_{i,j,n}(z)=\frac{P_{i,j,n}(z)}{Q_n(z)}=\langle(zI_n-H_n)^{-1} e_{j}, e_{i}\rangle
\]
is a rational function. Moreover, since $\langle (zI_n-H_{n})^{-1} e_{j}, e_{i}\rangle$ is the $(i,j)$-entry of $(zI_n-H_{n})^{-1}$, we can write
\[
 P_{i,j,n}(z) =(-1)^{i+j}\det((z I_n-H_{n})^{[j,i]}),
\]
where $(z I_n-H_{n})^{[j,i]}$ is the submatrix of $z I_n-H_{n}$ obtained by removing  the $j$-th row and the $i$-th column.
\par The following result may be proved in much the same way as Theorem \ref{theorem2}. We show that the resolvent functions $R_{i,j,n}(z)$ coincide with the power series $A_{i,j,n}(z)$ generated by the weight polynomials associated with the collections of lattice paths $\mathcal{D}_{[\ell, i, j, n]}$, $\ell\geq 0$.
\begin{theorem}\label{theorem12-12-22-T}
For each $n\geq 1$ we have
\begin{equation}\label{12-12-22-7}
R_{i,j,n}(z)=A_{i,j,n}(z), \qquad 0\leq i,j \leq n-1. 
\end{equation} 
\end{theorem}
\begin{proof}
 In view of \eqref{12-12-22-3} and \eqref{12-12-22-4}, we need to show that for every $\ell\geq 0$ we have 
\begin{equation}\label{12-12-22-8}
(H_n^{\ell})_{i,j}=\langle{{H}}_{n}^{\ell}e_{j},e_{i}\rangle=A_{[\ell,i,j,n]}.
\end{equation}
We can express the entries \eqref{def:entriesH} of the matrix $H_n$ as follows:
\[
\begin{cases}
h_{i,j}=a_{\min (i,j)}^{(j-i)}, & -p \leq j-i\leq q, \ \ 0\leq i, j \leq n-1,\\
h_{i,j}=0, & \mbox{otherwise}.
\end{cases}
\]
Let  $\ell$ be a natural number. Fix integers $i$ and $j$,  $0\leq i, j \leq n-1.$ Set  $i_0=i$ and $i_\ell= j$. Writing out the matrix multiplication explicitly, we have
\begin{equation} \label{12-12-22-12}
\langle H_n^{\ell}\, e_{j},e_{i}\rangle=(H_n^{\ell})_{i,j}=\sum_{i_{1}, \ldots, i_{\ell-1}}h_{i_{0},i_{1}} h_{i_{1},i_{2}}\cdots 
h_{i_{\ell-2},i_{\ell-1}} h_{i_{\ell-1},i_{\ell}},
\end{equation}
where the sum runs over all choices of $0\leq i_{1},\ldots,i_{\ell-1}\leq n-1$. So, we can write
\[
\langle H_n^{\ell} e_{j},e_{i}\rangle=\sum_{i_{1}, \ldots, i_{\ell-1}}a_{\min (i_{0},i_{1})}^{(i_{1}-i_{0})}a_{\min (i_{1},i_{2})}^{(i_{2}-i_{1})} \cdots
 a_{\min (i_{\ell-2},i_{\ell-1})}^{(i_{\ell-1}-i_{\ell-2})}a_{\min (i_{\ell-1},i_{\ell})}^{(i_{\ell}-i_{\ell-1})},
\]
where  
\begin{equation}\label{12-12-22-11}
-p\leq i_{k+1}-i_{k}\leq q, \qquad 0 \leq i_{k},i_{k+1}\leq n-1,\qquad \mbox{for all}\,\,\,0 \leq k \leq \ell-1.
\end{equation}
Note that 
$$a_{\min (i_{k},i_{k+1})}^{(i_{k+1}-i_{k})}$$
is the weight of the step that starts at  the point $(k,i_k)$ and ends at the point $(k+1,i_{k+1})$ (see the proof of Theorem \ref{theorem2}).
Thus, the product 
$$a_{\min (i_{0},i_{1})}^{(i_{1}-i_{0})}a_{\min (i_{1},i_{2})}^{(i_{2}-i_{1})} \cdots
 a_{\min (i_{\ell-2},i_{\ell-1})}^{(i_{\ell-1}-i_{\ell-2})}a_{\min (i_{\ell-1},i_{\ell})}^{(i_{\ell}-i_{\ell-1})},$$
where \eqref{12-12-22-11} holds, is the weight of a lattice path of length $\ell$ with initial point $(0,i)$, final point $(\ell,j)$, with no point below the $x$-axis or above the line $y=n-1$; that is, a path in $\mathcal{D}_{[\ell,i,j,n]}$. 
Now, considering that by \eqref{12-12-22-12} the expression $\langle H_n^{\ell}e_{j},e_{i}\rangle$
equals the sum of such products, and there is a one-to-one correspondence between paths in $\mathcal{D}_{[\ell,i,j,n]}$ and
 choices of $i_{1},\ldots,i_{\ell-1}$ satisfying \eqref{12-12-22-11}, we get \eqref{12-12-22-8}
followed by \eqref{12-12-22-7}.
\end{proof}

\par 
If $\gamma$ is a lattice path in the graph $\mathcal{G}$, we define $\max(\gamma)$ to be the maximum of the heights of all vertices in $\gamma$, and $\min(\gamma)$ to be the minimum of the heights of all vertices in $\gamma$. For example, for the path $\gamma$ in Fig.~\ref{7-1-22-2} we have $\min(\gamma)=-2$ and $\max(\gamma)=3$.

The following result demonstrates that the rational functions $R_{i,j,n}(z)$ serve as rational approximants of the formal power series $A_{i,j}(z)$. Specifically, it states that the coefficients of the power series expansions of $A_{i,j}(z)$ and $R_{i,j,n}(z)$ must be equal up to a certain order. This order tends to infinity as $n\rightarrow \infty$, asymptotically like 
$n(1/p + 1/q).$

\begin{theorem}\label{5-16-23-11}
Let $n \geq 1$ and $0 \leq i,j \leq n-1$ be integers. Set
\[
\displaystyle L := \left[\frac{n-1-i}{q}\right] + \left[\frac{n-1-j}{p}\right]+1,
\]
where $[\cdot]$ denotes the integer part. If $0\leq \ell \leq L$, then the two  collections of paths $\mathcal{D}_{[\ell,i,j,n]}$ 
and $\mathcal{D}_{[\ell,i,j]}$ coincide
\begin{equation}\label{5-16-23-6}
\mathcal{D}_{[\ell,i,j,n]}=\mathcal{D}_{[\ell,i,j]},
\end{equation}
and
\begin{equation}\label{5-18-23-1}
A_{[\ell,i,j,n]}=A_{[\ell,i,j]}.
\end{equation}
Moreover,
\begin{equation}\label{5-16-23-3}
A_{i,j}(z)-R_{i,j,n}(z)=O(z^{-L-2}), \qquad z\rightarrow \infty.
\end{equation}
\end{theorem}
\begin{proof} 
Under the given assumptions, we will show that \eqref{5-16-23-6} holds. Then equation  \eqref{5-18-23-1} follows from definition of the 
weight polynomials $A_{[\ell,i,j,n]}$ and $A_{[\ell,i,j]}.$
The validity of this equation directly implies the estimate \eqref{5-16-23-3}. Since $\mathcal{D}_{[\ell,i,j,n]}\subset\mathcal{D}_{[\ell,i,j]}$, 
it is enough to prove $\mathcal{D}_{[\ell,i,j]}\subset\mathcal{D}_{[\ell,i,j,n]}$. Therefore, let $\gamma\in\mathcal{D}_{[\ell,i,j]}$ be an arbitrary path. Our goal is to demonstrate that
\begin{equation}\label{5-16-23-4}
\max(\gamma) \leq n-1.
\end{equation}
By definition, this implies that $\gamma\in\mathcal{D}_{[\ell,i,j,n]}$, thereby justifying the desired inclusion. Using the division algorithm, we  can write
\begin{equation}
n-1-i=qm+q_1, \qquad 0\le q_1<q, \label{5-16-23-7}\\
\end{equation}
and 
\begin{equation}
 n-1-j=pk+p_1,\qquad  0\le p_1<p, \label{5-16-23-8}\\
\end{equation}
where $q_1 $ and $p_1$ are integers.
Therefore
\[
m=\left[\frac{n-1-i}{q}\right],\qquad
k=\left[\frac{n-1-j}{p}\right],
\]
and $L=m+k+1.$
Suppose that $\max(\gamma)>n-1$. Since $\gamma$ starts at the point $(0,i)$, equation \eqref{5-16-23-7} implies that $\gamma$ reaches its maximum height $\max(\gamma)$ in at least $m+1$ steps. After $\gamma$ reaches its maximum height, it must reach  height $j$ in at most $k$ steps, but this is impossible since
\[
\max(\gamma)-kp>n-1-kp=j+p_1 \ge j,
\]
where we used \eqref{5-16-23-8}. So \eqref{5-16-23-4} is justified. 
\end{proof}

Let $n$ be a positive integer. Consider the $q\times p$ matrix  
\begin{equation}\label{def:Rmatrix}
R_n(z):=\begin{pmatrix}
R_{0,0, n}(z) &\cdots & R_{0, p-1, n}(z)\\
\vdots &\ddots& \vdots\\
R_{q-1, 0, n}(z) & \cdots & R_{q-1, p-1, n}(z)\\
\end{pmatrix}
\end{equation}
with $(i,j)$-entry, $0 \leq i \leq q-1$, $0 \leq j \leq p-1$, given by the rational function $R_{i,j,n}(z)$. The explicit expression for $R_{i,j,n}(z)$ is provided in \eqref{12-12-22-4}. We present now the matrix continued fraction expansion for the matrix $R_n(z)$.
As the formula will show, the matrix $R_n(z)$ can be considered as a special truncation of the matrix continued fraction for $F(z)$ in \eqref{def:FCFmatrix} since the expansions \eqref{def:FCFmatrix} and \eqref{MCFforRn} share the following coefficients: $\alpha_{k}(z)$ for $0\leq k\leq n-1$, $\alpha_{k}^{+}$ for $0\leq k\leq n-q-1$, and $\alpha_{k}^{-}$ for $0\leq k\leq n-p-1$.
Before we describe the continued fraction for $R_{n}(z)$, we define several matrices $\rho_{k}(z), \rho_{k}^{+}, \rho_{k}^{-}$, $0\leq k\leq n-1$.

For any integer $0 \leq k \leq n-1$, the $q\times p$ matrix $\rho_{k}(z)$ is 
$$
\rho_{k}(z)=\alpha_{k}(z)=\begin{pmatrix}
0 &\cdots & 0 & 0\\
\vdots &\ddots& \vdots & \vdots\\
0 & \cdots & 0 & 0 \\
0 & \cdots & 0 & z-a^{(0)}_k
\end{pmatrix}.
$$
For $0 \leq k \leq n-q-1$, the $q\times q$ matrix $\rho^{+}_k$ is defined as
$$
\rho^{+}_k=\alpha^{+}_k=\begin{pmatrix}
1 & 0 & \cdots & 0 & 0\\
0 & 1 & \cdots & 0 & 0\\
\vdots & \vdots & \ddots & \vdots & \vdots \\
0 & 0 & \cdots & 1 & 0\\
-a^{(1)}_k & -a_{k}^{(2)} & \cdots & -a^{(q-1)}_k & -a_{k}^{(q)}
\end{pmatrix}.
$$
For $n-q\leq k \leq n-1$, the $q\times q$ matrix $\rho^{+}_k$ is given by
$$
\rho^{+}_k=\begin{pmatrix}
1 & 0 & \cdots & 0 & 0\\
0 & 1 & \cdots & 0 & 0\\
\vdots & \vdots & \ddots & \vdots & \vdots \\
0 & 0 & \cdots & 1 & 0\\
-b^{(1)}_k & -b_{k}^{(2)} & \cdots & -b^{(q-1)}_k & -b_{k}^{(q)}
\end{pmatrix},
$$
where the coefficients $b^{(i)}_{k}$ in the above matrix are defined as follows:
\[
\begin{cases}
b^{(i)}_{k}=a^{(i)}_{k}, & 1\leq i\leq n-k-1, \\[0.3em]
b_{k}^{(i)} =0, & n-k\leq i\leq q.\\[0.3em]
\end{cases}
\]
Additionally, for $0 \leq k \leq n-p-1$, the $p\times p$ matrix $\rho^{-}_k$ is 
$$
\rho^{-}_k=\alpha^{-}_k=\begin{pmatrix}
1 & 0 & \cdots & 0 & a_{k}^{(-1)}\\[0.2em]
0 & 1 & \cdots & 0 & a_{k}^{(-2)}\\
\vdots & \vdots & \ddots & \vdots & \vdots \\
0 & 0 & \cdots & 1 & a_{k}^{(-p+1)}\\[0.1em]
0 & 0 & \cdots & 0 & a_{k}^{(-p)}
\end{pmatrix}. 
$$
For $n-p\leq k \leq n-1$, the $p\times p$ matrix $\rho^{-}_k$ is given by
$$
\rho^{-}_k=\begin{pmatrix}
1 & 0 & \cdots & 0 & b_{k}^{(-1)}\\[0.2em]
0 & 1 & \cdots & 0 & b_{k}^{(-2)}\\
\vdots & \vdots & \ddots & \vdots & \vdots \\
0 & 0 & \cdots & 1 & b_{k}^{(-p+1)}\\[0.1em]
0 & 0 & \cdots & 0 & b_{k}^{(-p)}
\end{pmatrix}
$$
where the coefficients $b^{(-j)}_{k}$ in the above matrix are defined according to the following formulas
\[
\begin{cases}
b^{(-j)}_{k}=a^{(-j)}_{k}, & 1\leq j\leq n-k-1, \\[0.3em]
b_{k}^{(-j)} =0, & n-k\leq j\leq p.\\[0.3em]
\end{cases}
\]

Analogously to \eqref{FLTalphas}, we define the following transformations for matrices $X$ of size $q\times p$:
\[
\tau_{\rho,k}(X)=\frac{\mathbf{1}}{\rho_{k}(z)+\rho_{k}^{+}\,X\,\rho_{k}^{-}},\qquad 0\leq k\leq n-1.
\]

\begin{proposition}\label{prop:MCFforRn}
The matrix $R_{n}(z)$ defined in \eqref{def:Rmatrix} has the expansion
\[
R_{n}(z)=(\tau_{\rho,0}\circ\tau_{\rho,1}\circ\cdots\circ\tau_{\rho,n-1})(X_{n}(z)),
\]
where $X_{n}(z)$ is the $q\times p$ matrix with entries on the main diagonal equal to $1/z$, and zero entries elsewhere.
More graphically, we have
\begin{equation}\label{MCFforRn}
R_n(z)=\cfrac{\mathbf{1}}{\rho_{0}(z)+\rho_{0}^{+}\cfrac{\mathbf{1}}{\rho_{1}(z)+\rho_{1}^{+}\cfrac{\mathbf{1}}{\ddots\rho_{n-2}(z)+\rho_{n-2}^{+}\cfrac{\mathbf{1}}{\rho_{n-1}(z)+\rho_{n-1}^{+}X_{n}(z)\rho_{n-1}^{-}}\,\rho_{n-2}^{-}}\,\rho_{1}^{-}}\,\rho_{0}^{-}}.
\end{equation} 
\end{proposition}
\begin{proof}
Let $\widetilde{H}=(\widetilde{h}_{i,j})_{i,j=0}^{\infty}$ be the one-sided infinite matrix whose principal $n\times n$ truncation is the matrix $H_{n}$ in \eqref{10-26-22-1}, with all other entries of $\widetilde{H}$ being zero. In virtue of \eqref{12-12-22-4}, the matrix \eqref{def:Rmatrix} can be viewed as a matrix $F(z)$ of resolvent functions of the operator with matrix representation $\widetilde{H}$. Therefore, we can apply in this setting the same relations \eqref{10-21-22-2} between resolvent functions associated with operators obtained by deleting the first few rows and columns of $\widetilde{H}$, as it is done in the proof of Theorem~\ref{TH6-8-23-1}. After applying these relations $n$ times we obtain \eqref{MCFforRn}, where $X_{n}(z)$ is a matrix of resolvent functions for the operator identically zero. It is clear that such a matrix is  a $q\times p$ matrix with entries on the main diagonal equal to $1/z$, and zero entries elsewhere.
\end{proof}

In the next section, we will use the following proposition.
\begin{proposition}\label{12-12-22-30}
Let $\ell\geq 1$ be a positive integer and let 
\begin{equation}\label{12-12-22-31}
\displaystyle N:=\left[\frac{p q \ell}{p+q}\right]+\max\{p,q\}+1, 
\end{equation}
where $[\cdot]$ is the integer part. Suppose that $n\geq 2N+1$ and $ N \leq i\leq n-1-N $. Then there is a one-to-one correspondence between paths in $\mathcal{D}_{[\ell,i,i,n]}$ and paths in $\mathcal{P}_{[\ell,0,0]}.$ The paths in the collection $\mathcal{D}_{[\ell,i,i,n]}$ can be obtained by shifting the paths in the collection $\mathcal{P}_{[\ell,0,0]}$ $i$ units upwards.
\end{proposition}
\begin{proof}
We will show that under the stated assumptions we have $\mathcal{D}_{[\ell,i,i,n]}=\mathcal{P}_{[\ell,i,i]}$. This identity clearly justifies the claim. Since by definition $\mathcal{D}_{[\ell,i,i,n]}\subset\mathcal{P}_{[\ell,i,i]}$, it suffices to prove $\mathcal{P}_{[\ell,i,i]}\subset\mathcal{D}_{[\ell,i,i,n]}$. So let $\gamma\in\mathcal{P}_{[\ell,i,i]}$ be arbitrary. We will show that
\begin{align}
\max(\gamma) & \leq i+N, \qquad \label{12-12-22-34}\\
\min(\gamma) & \geq i-N. \qquad \label{12-12-22-35}
\end{align}
Since by assumption $i-N\geq 0$ and $i+N\leq n-1$, it follows that $\gamma\in\mathcal{D}_{[\ell,i,i,n]}$, hence the desired inclusion will be justified.

First we prove \eqref{12-12-22-34}. Let $m$ be the largest integer such that $qm \leq (\ell-m)p$. This implies that $0\leq m<\ell$ and 
\begin{equation}\label{ineqpq1}
q(m+1)>(\ell-m-1)p.
\end{equation}
Then we can write
\begin{equation}\label{12-12-22-36}
qm=(\ell-m)p-p_1,
\end{equation}
where $p_1$ is an integer satisfying $0\leq p_1 < p+q.$ Solving for $m$ in \eqref{12-12-22-36} we get
\begin{equation}\label{definm} 
m=\frac{\ell p-p_1}{p+q}.
\end{equation}
We claim that
\begin{equation}\label{ineqpq2}
\max(\gamma)\leq i+q(m+1).
\end{equation}
Indeed, suppose that $\max(\gamma)>i+q(m+1)$. Since $\gamma$ starts at the point $(0,i)$, $\gamma$ reaches its maximum height $\max(\gamma)$ in at least $m+2$ steps. After $\gamma$ reaches its maximum height, it must return to height $i$ in at most $\ell-m-2$ steps, but this is impossible since
\[
\max(\gamma)-(\ell-m-2)p>i+q(m+1)-(\ell-m-2)p>i,
\]
where we used \eqref{ineqpq1}. So \eqref{ineqpq2} is justified. 

Applying \eqref{ineqpq2}, \eqref{definm}, and \eqref{12-12-22-31} we obtain
$$
\max(\gamma)-i\leq q(m+1)=\frac{\ell pq-p_1q}{p+q}+q \leq \frac{\ell pq}{p+q}+q\leq N.
$$

\par We can prove \eqref{12-12-22-35} in a similar way. Now let $m$ be the largest integer such that $pm \leq (\ell-m)q.$ As above, we have $0\leq m<\ell$ and
\[
p(m+1)>(\ell-m-1)q.
\]
Writing 
\[
pm=(\ell-m)q-q_1,  
\]
we have $0\leq q_1 < p+q$ and 
$$ 
m=\frac{\ell q-q_1}{p+q}.
$$
The reader can check that $\min(\gamma)\geq i-p(m+1)$, and therefore
$$
i-\min(\gamma)\leq p(m+1)=\frac{\ell pq-q_1p}{p+q}+p\leq N.   
$$
\end{proof}

\section{Random polynomials }\label{randompol}

In this section we assume that the banded matrices $H$ and $W$ have $p+q+1$ diagonals with entries that are independent random variables. 
\par
Let $\mu_{k}$, $-p\leq k\leq q$, be a collection of $p+q+1$ Borel probability measures with compact support in $\mathbb{C}$.
For each $-p\leq k\leq q$, let $a^{(k)}=(a^{(k)}_{n})_{n \in \mathbb{Z}}$ be a sequence of i.i.d. random variables with distribution $\mu_{k}$. Moreover, we assume that the whole collection $\{a^{(k)}_{n}: n \in \mathbb{Z}, -p\leq k\leq q\}$ is jointly independent
and the entries $a^{(k)}_{n}$ are surely bounded in modulus by an absolute constant.
\par
Let $n$ be a nonnegative integer. As in Section \ref{12-12-22-1}, let  $H_{n}$ be the principal $n\times n$ truncation of the banded matrix $H$ and let $Q_{n}(z)=\det(z I_{n}-H_{n})$ be the characteristic polynomial of the matrix $H_{n}$.
\par
Denote by $\{\lambda_{i,n}\}_{i=1}^{n}$
the eigenvalues of $H_{n}$, counting multiplicities. 
Let $\sigma_{n}$ be the empirical measure  of the matrix $H_{n}$
\[
\sigma_{n}:=\frac{1}{n}\sum_{i=1}^{n}\delta_{\lambda_{i,n}}.
\]
Since we have uniform boundedness of the matrix entries, the eigenvalues $\lambda_{i,n}$ are also uniformly bounded. Clearly, $\sigma_{n}$ is a random probability measure. Its mean $\mathbb{E}\sigma_{n}$ is the probability measure defined via duality by
\[
\int g\,d\mathbb{E}\sigma_{n}=\mathbb{E} \int g\,d\sigma_{n}
\]
for every continuous function $g$.

\par
Theorem \ref{10-26-2022} gives an expression for the limit of the expected values of the moments of the measure $\sigma_n$. In the case that all the eigenvalues of $H_n$ are real for every $n$, Theorem \ref{10-26-2022} implies that
the average measure $\mathbb{E}\sigma_n$ converges weakly to a probability distribution on the real line.

\begin{theorem}\label{10-26-2022}
Let $\mu_{k}$, $-p\leq k\leq q$, be a collection of $p+q+1$ arbitrary Borel probability measures with compact support in $\mathbb{C}$.
For each $-p\leq k\leq q$, let $a^{(k)}=(a^{(k)}_{n})_{n \in \mathbb{Z}}$ be a sequence of i.i.d. random variables with distribution $\mu_{k}$. Assume also that the whole collection $\{a^{(k)}_{n}: n \in \mathbb{Z}, -p\leq k\leq q\}$ is jointly independent
and the entries $a^{(k)}_{n}$ are surely bounded in modulus by an absolute constant. Then for each nonnegative integer $\ell$ we have
\begin{equation}\label{10-26-22-10}
\lim_{n\rightarrow\infty} \mathbb{E} (\int z^{\ell} d\sigma_{n}(z))=\lim_{n\rightarrow\infty}\mathbb{E}\left(\frac{1}{n}\sum_{i=1}^{n}
\lambda_{i,n}^{\ell}\right)=\mathbb{E} (W_{[\ell,0,0]}).
 \end{equation}
\end{theorem}
Since we have uniform boundedness of the supports of the measures $\mathbb{E}\sigma_{n}$, an equivalent
formulation of \eqref{10-26-22-10} is that for all $z$ large enough,
\[
\lim_{n\rightarrow\infty} \mathbb{E} (\int \frac{d\sigma_{n}(x)}{z-x})=\lim_{n\rightarrow\infty}\mathbb{E}\left(\frac{1}{n}\sum_{i=1}^{n}
\frac{1}{z-\lambda_{i,n}}\right)= \mathbb{E} (W_{0,0}(z)).
 \]
\begin{proof}
Let $\ell\geq 0$ be fixed. Since 
\[
\int x^{\ell}\, d\sigma_{n}(x)=\frac{1}{n}\,\Tr(H_{n}^{\ell}),
\]
we will establish the limit
\begin{equation}\label{12-13-22-1}
\lim_{n\rightarrow\infty}\frac{1}{n}\mathbb{E}(\Tr(H_{n}^{\ell}))=\mathbb{E} (W_{[\ell, 0.0]}).
\end{equation}
Let us first write
\begin{equation}\label{12-12-22-42}
\Tr(H_{n}^{\ell})=\sum_{i=0}^{n-1}\langle H_n^{\ell}\, e_{i},e_{i}\rangle=\sum_{i=0}^{n-1}(H_n^{\ell})_{i,i}.
\end{equation}
We represent the last sum in the form
\begin{equation}\label{12-12-22-43}
\sum_{i=0}^{n-1}(H_n^{\ell})_{i,i}=\sum_{ N \leq i\leq n-1-N }(H_n^{\ell})_{i,i}+S_n,
\end{equation}
where
\begin{equation}\label{12-12-22-44}
S_n:=
\sum_{0 \leq i < N }(H_n^{\ell})_{i,i}+
\sum_{n-1-N < i\leq n-1}(H_n^{\ell})_{i,i},
\end{equation}
and 
\[
\displaystyle N=\left[\frac{pq\ell}{p+q}\right]+\max\{p,q\}+1.
\]
Let $N \leq i\leq n-1-N.$ Given that the random variables $\{a_{n}^{(k)}: n\in\mathbb{Z}, -p\leq k\leq q\}$ are independent, and for each $k$ the variables $(a_{n}^{(k)})_{n\in\mathbb{Z}}$ are identically distributed, by
 \eqref{12-12-22-8} and
with the help of Proposition \ref{12-12-22-30} we get
$$ \mathbb{E}(H_n^{\ell})_{i,i}=\mathbb{E} (W_{[\ell,0,0]}).$$
Furthermore, since the random variables $a_{n}^{(k)}$ are surely bounded in modulus by an absolute constant, we obtain that the absolute value of each $(H_n^{\ell})_{i,i}$ in \eqref{12-12-22-44} is bounded by a constant that only depends on $\ell$ (not on $n$).
 Therefore, by \eqref{12-12-22-42}, \eqref{12-12-22-43}, and\eqref{12-12-22-44},
\[
\frac{1}{n}\mathbb{E}(\Tr(H_{n}^{\ell}))=\frac{n-2N}{n}\mathbb{E} (W_{[\ell,0,0]}) +o(1),\quad n\rightarrow \infty,
\]
and then, 
\eqref{12-13-22-1} and \eqref{10-26-22-10} follow.
\end{proof}

\bigskip

\noindent \textsc{The University of Alabama,
Tuscaloosa, AL 35487, USA} \\
\textit{Email address}: \texttt{jdkim4\symbol{'100}crimson.ua.edu}

\bigskip

\noindent \textsc{Department of Mathematics, University of Central Florida, 4393 Andromeda Loop North, Orlando, FL 32816, USA} \\
\textit{Email address}: \texttt{abey.lopez-garcia\symbol{'100}ucf.edu}

\bigskip

\noindent \textsc{Department of Mathematics and Statistics, University of South Alabama, 411 University Boulevard North, 
Mobile, AL 36688, USA}\\
\textit{Email address}: \texttt{prokhoro\symbol{'100}southalabama.edu}

\end{document}